\documentclass[10pt,a4paper,oneside,reqno]{amsart}	

\usepackage[utf8]{inputenc}
\usepackage[english]{babel}
\usepackage{amsfonts}
\usepackage{amsmath}
\usepackage{amssymb}
\usepackage{amsthm}
\usepackage{tikz-cd}
\usepackage{xpatch}

\usepackage[backend=biber,
style=alphabetic,
sorting=nyt,
backref,
backrefstyle=none,
maxnames=4,
maxbibnames=999]{biblatex}
\usepackage{hyperref}
\hypersetup{
	citecolor=green,
	draft=false,
	colorlinks, 
	pdfauthor={Enrico Trebeschi},
	pdftitle={Quantitative estimates on principal curvatures of constant mean curvature spacelike hypersurfaces in Anti-de Sitter space.},
	linktocpage
}

\theoremstyle{plain}
\newtheorem{thm}{Theorem}[section]
\newtheorem{pro}[thm]{Proposition}
\newtheorem{cor}[thm]{Corollary}
\newtheorem{lem}[thm]{Lemma}

\newtheorem*{thm*}{Theorem}

\makeatletter
\newtheorem*{rep@theorem}{\rep@title}
\newcommand{\newreptheorem}[2]{%
	\newenvironment{rep#1}[1]{%
		\def\rep@title{#2 \ref{##1}}%
		\begin{rep@theorem}}%
		{\end{rep@theorem}}}
\makeatother

\newtheorem{thmx}{Theorem}
\newreptheorem{thmx}{Theorem}

\newtheorem{prox}[thmx]{Proposition}
\newreptheorem{prox}{Proposition}

\newtheorem{corx}[thmx]{Corollary}
\newreptheorem{corx}{Corollary}

\newreptheorem{lemx}{Lemma}

\theoremstyle{remark}
\newtheorem{rem}[thm]{Remark}
\newtheorem{de}[thm]{Definition}

\makeatletter
\newcounter{step}
\xpretocmd{\proof}{\setcounter{step}{0}}{}{}
\newcommand{\step}[1]{%
	\par
	\addvspace{\medskipamount}%
	\stepcounter{step}%
	\noindent\emph{Step \thestep: #1.}\par\nobreak\smallskip
	\@afterheading
}
\makeatother

\newcommand{\N}{\mathbb{N}}

\newcommand{\R}{\mathbb{R}}
\newcommand{\C}{\mathbb{C}}
\newcommand{\pj}{\mathbb{P}}

\newcommand{\psl}{\mathbb{P}\mathrm{SL}(2,\mathbb{R})}
\newcommand{\hyp}{\mathbb{H}}
\newcommand{\hypu}{\widetilde{\mathbb{H}}}
\newcommand{\sph}{\mathbb{S}}
\newcommand{\ch}{\mathcal{CH}}

\newcommand{\cmclp}{\mathcal{CMC}}
\newcommand{\pd}{\partial}
\newcommand{\sq}{\subseteq}
\newcommand{\pr}[1]{\langle#1\rangle}

\newcommand{\sff}{\mathrm{I\!I}}
\newcommand{\Id}{\mathrm{Id}}
\newcommand{\past}{\mathbf{P}}
\newcommand{\fut}{\mathbf{F}}
\newcommand{\nablah}{\bar{\nabla}}
\newcommand{\Span}{\mathrm{Span}}
\newcommand{\tr}{\operatorname{tr}}
\newcommand{\gr}{\operatorname{graph}}
\newcommand{\hess}{\operatorname{Hess}}
\newcommand{\Gau}{\mathcal{G}}
\DeclareMathOperator{\dist}{dist}
\DeclareMathOperator{\arcsinh}{arcsinh}

\title[$H-$convexity and estimates for CMC hypersurfaces in $\mathbb{H}^{n,1}$]{Generalized convexity and quantitative estimates for constant mean curvature spacelike hypersurfaces in Anti-de Sitter space.}
\author[Enrico Trebeschi]{Enrico Trebeschi}
\address{Laboratoire J.A. Dieudonn\'e\\
	Université Côte d'Azur\\
	France}
\email{enrico.trebeschi@univ-cotedazur.fr}

\thanks{E.~Trebeschi acknowledges funding by the European Research Council under ERC-Advanced grant 101095722 AnSur (Geometric Analysis and Surface Groups), ERC-consolidator grant 101124349 GENERATE (GeomEtry and aNalysis for $(G,X)-$structurEs and their
	defoRmATion spacEs). Views and opinions expressed are however those of the author(s) only and do not necessarily reflect those of the European Union or the European Research Council Executive Agency. Neither the European Union nor the granting authority can be held responsible for them. This work has been partially funded by ANR JCJC grant GAPR (ANR-22-CE40-0001, Geometry and Analysis in the Pseudo-Riemannian setting), and PRIN F53D23002800001.}
	
\begin{document}
	\begin{abstract}
		We study the principal curvatures of properly embedded constant mean curvature hypersurfaces in the Anti-de Sitter space $\mathbb{H}^{n,1}$. We generalize the notion of convex hull and give an upper bound on the principal curvatures which only depends on the width of the $H-$shifted convex hull. This analysis has two direct consequences. First, it allows to bound the sectional curvature of $H-$hypersurfaces by an explicit function of the the width of the $H-$shifted convex hull. Second, we bound the quasiconfromal dilatation of a class of quasiconformal maps on the hyperbolic plane $\mathbb{H}^2$, called $\theta-$landslides, in terms of the cross-ratio norm of their quasi-symmetric extension on $\partial_\infty\mathbb{H}^2$.
	\end{abstract}
	
	\maketitle
	
	\tableofcontents
	\section{Introduction}
	The Anti-de Sitter space $\hyp^{n,1}$ is a Lorentzian manifold with constant sectional curvature $-1$. It identifies with the space of oriented negative lines of a non-degenerate bilinear form of signature $(n,2)$ and it admits a conformal asymptotic boundary $\pd_\infty\hyp^{n,1}$, consisting of oriented degenerate lines of such bilinear form, which is diffeomorphic to $\sph^{n-1}\times\sph^1$. The space $\hyp^{n,1}\cup\pd_\infty\hyp^{n,1}$ conformally compactifies the Anti-de Sitter space.
	
	An immersed hypersurface is \textit{spacelike} if the induced metric is Riemannian. A properly embedded spacelike hypersurface $\Sigma$ in $\hyp^{n,1}$ has a well defined asymptotic boundary
	\[\pd_\infty\Sigma:=\overline{\Sigma}\cap\pd_\infty\hyp^{n,1},\]
	for $\overline{\Sigma}$ the topological closure of $\Sigma$ inside $\hyp^{n,1}\cup\pd_\infty\hyp^{n,1}$. It turns out that $\pd_\infty\Sigma$ is the graph of a $1-$Lipschitz map from $\sph^n$ to $\sph^1$ containing no antipodal points. We call the graph of such a map an \textit{admissible boundary} (see Definition~\ref{de:adm}). 
	
	Properly embedded \textit{constant mean curvature} (CMC) spacelike hypersurfaces in Anti-de Sitter space are uniquely determined by the value $H\in\R$ of their mean curvature and their admissible asymptotic boundary $\Lambda$ by \cite[Theorem~A]{ecrin}:
	\begin{thm*}
		For any admissible asymptotic boundary $\Lambda$ in $\pd_\infty\hyp^{n,1}$ and for any real number $H\in\R$, there exists a unique properly embedded hypersurface $\Sigma$ such that $\pd_\infty\Sigma=\Lambda$ with constant mean curvature $H$.
	\end{thm*}
	We want to quantify this dependence by controlling the geometry of the CMC hypersurface in terms of the data $(H,\Lambda)$.
	
	\subsection{Historical background}
	The study of \textit{constant mean curvature} (CMC) hypersurfaces is at the crossroad of several branches of mathematics: in differential geometry is motivated by their geometrical features and in calculus of variations since they are critical points for area-like operators. In general relativity, \textit{spacelike} CMC hypersurfaces are preferred Cauchy data for Einstein equations and in Lorentzian geometry they induce natural time coordinates, useful to study the geometry of the ambient manifold (see for example \cite{fol3,fol1,fol2}). These motivations are behind a vast literature on spacelike CMC hypersurfaces in general Lorentzian manifolds (see for example \cite{cho,ecker,bartnikH,bartnik}).
	
	Spacelike CMC hypersurfaces in \textit{Lorentzian spaceforms} have been fully classified: for the flat case, \textit{i.e.} the Minkowski space $\R^{n,1}$, the program traces back to the work of \cite{cheng-yau}, followed by \cite{choi-trei,trei,abbz}, until the complete classification in \cite{bss}. For the negatively curved case, \textit{i.e.} the Anti-de Sitter space $\hyp^{n,1}$, the classification started with \cite{bbz}, followed by \cite{abbz,tamb}, and it has been recently completed by \cite{ecrin} by the statement above.
	
	Within the framework of geometric topology, spacelike CMC hypersurfaces have become an invaluable tool. The study of surfaces in $\R^{2,1}$ and $\hyp^{2,1}$ has become of great interest since the pioneering work of Mess \cite{mess}, mostly because of their relation with Teichm\"{u}ller theory (see \cite{bbz,benbon,univ,tamb,andreamax,bonsep}). In higher dimension, see \cite{abbz,barmer}.
	
	In higher co-dimension, CMC hypersurfaces generalize to \emph{parallel} mean curvature spacelike $n-$submanifold. Further progress have been made on specific pseudo-Riemannian spaces, notably on the so-called \textit{indefinite space-forms} of signature $(n,k)$: the pseudo-Euclidean space $\R^{n,k}$, the pseudo-hyperbolic space $\hyp^{n,k}$ and pseudo-spherical space $\sph^{n,k}$. See the works \cite{ishi,kkn,pmcbound,h2n,lt,mt} for estimates on the geometry of parallel mean curvature spacelike $n-$submanifolds in this setting. Recently, maximal spacelike $n-$submanifolds in $\hyp^{n,k}$ have been studied in relation with higher higher Teichm\"{u}ller theory (see \cite{dgk,ctt,ltw,lt,sst,beykas,mv,mt}).
	
	The natural question of relating the geometry of maximal hypersurfaces with the data $(H,\Lambda)$ has been adressed for $H=0$ and $n=2$. It is known that maximal surfaces in $\hyp^{2,1}$ has non-positive sectional curvature: from a \textit{qualitative} point of view, \cite{univ} gives a criterion to distinguish if a maximal surface is uniformly negatively curved or not only in term of its asymptotic boundary. An analogous result has been achieved in \cite{lt} in the broader setting of the pseudo-hyperbolic space $\hyp^{2,k}$, for any $k\ge1$. From a \textit{quantitative} point of view, \cite{andreamax} estimates the extrinsic curvature of maximal surfaces in $\hyp^{2,1}$ in terms of their convex hulls. The recent work \cite{alex} studies the subclass of admissible boundaries of $\hyp^{2,1}$ consisting of \textit{lightlike polygons}, \textit{i.e.} graphs of piecewise isometries of $\sph^1$, and estimates the total area of the corresponding maximal surface in terms of the number of the vertex of the polygon.
	
	\subsection{$H-$convexity} We introduce the notion of $H-$convexity (Definition~\ref{de:Hconv}), generalizing the usual notion of convexity. The \textit{$H-$shifted convex hull} of $\Lambda$ is the smallest $H-$convex subset of $\hyp^{n,1}$ containing $\Lambda$, denoted by $\ch_H(\Lambda)$ (see Definition~\ref{de:Hhull}).
	
	For $H=0$, the $H-$shifted convex hull of $\Lambda$ coincides with the usual convex hull, namely the intersection of halfspaces containing $\Lambda$. For $H\ne0$, we generalize the definition for making it more suitable for $H-$hypersurfaces. More precisely, a halfspace is one of connected component of the complement of a totally geodesic hypersurface: we call a $H-$halfspace one connected component of the complement of a $H-$umbilical hypersurface. The $\ch_H(\Lambda)$ is then the intersection of all the $H-$halfspaces containing $\Lambda$ (see Figure~\ref{fig:Hconvex}). This definition generalizes the notion of $H-$shifted convex hull introduced by \cite{coshshift} for the hyperbolic space.
	
		\begin{figure}[h]
		\centering
		\begin{minipage}[c]{.4\textwidth}
			\centering
			\includegraphics[width=\textwidth]{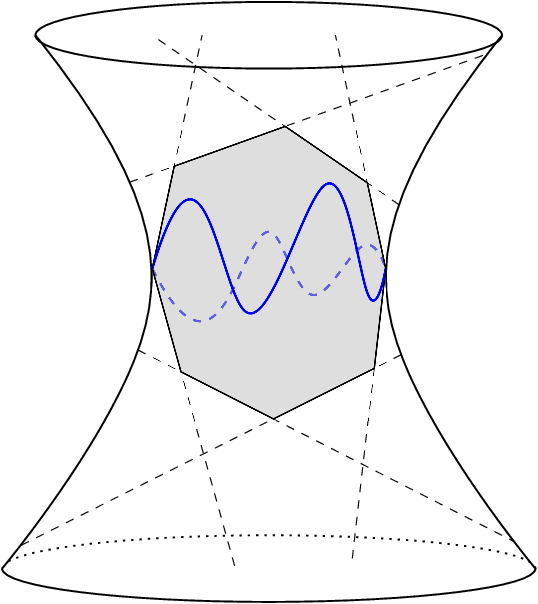}
		\end{minipage}%
		\hspace{1.5cm}
		\begin{minipage}[c]{.4\textwidth}
			\centering
			\includegraphics[width=\textwidth]{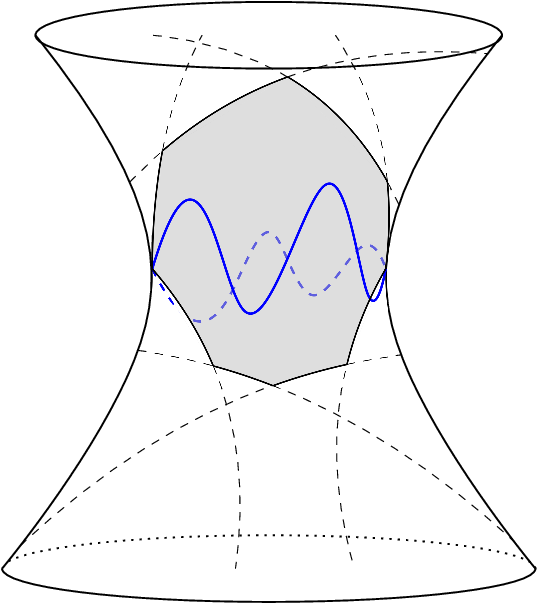}
		\end{minipage}%
		\caption{The shaded areas represent the convex hull (left) and the $H-$convex hull (right) of the blue curve. The dashed lines represent the support totally geodesic (resp. totally umbilical) hypersurfaces.}\label{fig:Hconvex}
	\end{figure}
	
	The \textit{width} of a subspace in the Anti-de Sitter space is defined as its timelike diameter. The $H-$width $\omega_H(\Lambda)$ of an admissible boundary $\Lambda$ is the width of its $H-$shifted convex hull.
	
	We begin by establishing some preliminary properties which will be of great importance in our main results.
	
	We give an explicit sharp upper bound for the $H-$width, seen as a function $\omega_H(\cdot)$ on the space of admissible boundaries (Corollary~\ref{cor:width}). Conversely, for a fixed $\Lambda$, we study the function $\omega_\bullet(\Lambda)$ defined on $\R$: 
		\begin{prox}\label{pro:width}
		Let $\Lambda$ be an admissible boundary in $\pd_\infty\hyp^{n,1}$. The function \[\omega_\bullet(\Lambda)\colon\R\to[0,\pi/2]\] has the following properties:
		\begin{enumerate}
			\item it is continuous;
			\item it is increasing (resp. decreasing) for $H\le0$ (resp. for $H\ge0$);
			\item it achieves its maximum at $H=0$;
			\item $\lim_{H\to\pm\infty}\omega_\bullet(\Lambda)=0$.
		\end{enumerate}
	\end{prox}
	Moreover, in Lemma~\ref{lem:width} we describe the relation between the $\omega_H(\Lambda)$ and $\omega_{H'}(\Lambda)$, for $H\ne H'$.

	By the strong maximum principle for CMC hypersurfaces (see \cite[Proposition~5.2.1]{ecrin}, the CMC hypersurface $\Sigma$ determined by the data $(H,\Lambda)$ is contained in $\ch_H(\Lambda)$ (Corollary~\ref{cor:Hmaxprin}) generalizing \cite[Theorem~3.2]{coshshift}.
	
	The hypersurface $\Sigma$ is trapped in such portion of space: roughly speaking, we prove in Theorem~\ref{thm:Schauder} (see next subsection) that the extrinsic curvature of $\Sigma$ is controlled by the \textit{width} $\omega_H(\Lambda)$ (Definition~\ref{de:width}) of $\ch_H(\Lambda)$, namely its timelike diameter.
	
	\subsection{Upper bound} We prove that the width of the $H-$shifted convex hull is an upper bound for the extrinsic curvature, quantified by the norm of the traceless shape operator.
	\begin{thmx}\label{thm:Schauder}
		Let $L\ge K\ge0$. There exists a universal constant $C_L$ with the following property: let $\Sigma$ a properly embedded $H-$hypersurface in $\hyp^{n,1}$ with $H\in[K,L]$, and let $B_0$ be its traceless shape operator. Then,
		\[\|B_0\|_{C^0(\Sigma)}\le C_L\sin\left(\omega_K(\Lambda)\right),\]
		for $\omega_K$ the width of the $K-$shifted convex hull of the asymptotic boundary $\Lambda$ of $\Sigma$.
	\end{thmx}
	This result generalizes \cite[Theorem~1.A]{andreamax}, which focuses on maximal surfaces in $\hyp^{2,1}$, to any dimension and any values of mean curvature $H\in\R$.
	
	\begin{rem}\label{rem:shauder}
		The interest of this result lies close to the Fuchsian locus, namely for small values of $\omega_K$. This is the case if $\Lambda$ is closed to be the boundary of a totally umbilical hypersurface.
		
		The uniform bound on the norm of the shape operator (or, equivalently, of the second fundamental form) of $H-$hypersurfaces implied by substituting in Theorem~\ref{thm:Schauder} the maximum of $\omega_K(\cdot)$ (see Lemma~\ref{lem:width}) was already proved in \cite[Theorem~6.2.1]{ecrin}: in fact, the existence of a uniform bound is one of the key points in the proof of Theorem~\ref{thm:Schauder}, as we will explain in Subsection~\ref{sec:sketch}.
	\end{rem} 
	
	\subsection{Sectional curvature} Our first application of Theorem~\ref{thm:Schauder} is devoted to the study of the intrinsic curvature of CMC hypersurfaces.
	
	\begin{corx}\label{cor:sectional}
		For any $H\in\R$, there exists a universal constant $K_H>0$ such that, for any properly embedded $H-$hypersurface $\Sigma$, it holds
		\[K_\Sigma\le-1-\left(\frac{H}{n}\right)^2+K_H\sin\left(\omega_H(\Lambda)\right),\]	
		for $\Lambda$ the asymptotic boundary of $\Sigma$.
	\end{corx}
	For $n=2$, CMC surfaces in $\hyp^{2,1}$ are known to be non-positively curved, as proved in\cite[Lemma~10.9]{particles} for $H=0$ and in \cite[Lemma~3.6]{tamb} for arbitrary values of $H\in\R$. Moreover, maximal surfaces in $\hyp^{2,1}$ are uniformely negatively curved precisely when the width of their convex hull is strictly lower than $\pi/2$ by \cite[Corollary~3.22]{univ}: for $n=2$, Corollary~\ref{cor:sectional} quantifies the relation between the intrinsic geometry of CMC surfaces and their asymptotic data (compare with Corollary~\ref{lem:sect}).
	
	The result is way more innovative in higher dimension, where nothing is known about the sectional curvature of CMC hypersurfaces, to the best of our knowledge. A first consequence is that it produces a large class of uniformly negatively curved CMC hypersurfaces: this fact corroborates the idea that all CMC hypersurfaces in $\hyp^{n,1}$ are Hadamard manifolds, \textit{i.e.} non-positively curved.
	
	In fact, a natural question is whether a closed manifold $S$ admits a spacelike embedding as a Cauchy hypersurface in a globally hyperbolic Anti-de Sitter manifold. Equivalently, whether $S\times\R$ admits an Anti-de Sitter structure, \textit{i.e.} a Lorentzian metric of constant sectional curvature $K=-1$ such that the fibers over $S$ are timelike curves.
	
	If $S\times\R$ carries an Anti-de Sitter structure, then $\tilde{S}\times\R$ embeds as an open set of $\hyp^{n,1}$, and $\tilde{S}\times\{t\}$ is a properly embedded spacelike hypersurface (see \cite[Proposition~6.3.1, Corollary~6.3.13]{benbon}), for any $t\in\R$. Moreover, it turns out that one can assume $S\times\{H\}$ to be an $H-$hypersurface, for any value of $H\in\R$ (see \cite[Theorem~1.1]{bbz} for $n=2$, \cite[Theorem~1.7]{abbz} for $n>2$ and \cite[Corollary~D]{ecrin} for the non-equivariant case), making the study of CMC hypersurfaces a preferred setting to attack this problem. 
	
	The case when $n=2$ is well understood, using the geometry of CMC surfaces: globally hyperbolic Cauchy compact Anti-de Sitter manifolds are of the form $S\times\R$ for $S$ a hyperbolic surface or a flat torus. In the hyperbolic case, the moduli space of Anti-de Sitter manifolds modeled on $S$ is isomorphic to the product of two copies of the Teichm\"uller space of $S$, while for the flat case it is corresponds precisely to the Teichm\"uller space of $S$. 
	
	In higher dimension, it remains an open question. Trivially, any hyperbolic manifold can be embedded as a spacelike Cauchy hypersurface in a globally hyperbolic Anti-de Sitter manifold. Exotic examples have been produced: for $n=4,5,6,7,8$, \cite{leemarquis} shows that $S$ can be chose to be not quasi-isometric to $\hyp^{n}$ (nor to any other symmetric space); for any $n>3$, \cite{mst} constructs examples where $S$ is a Gromov-Thurston manifold, \textit{i.e.} a manifold which can carry a negatively curved metric but not a hyperbolic one. 
	
	In the opposite direction, \cite[Corollary~1.2]{sst} gives a topological constraint: the manifold $S$ has to admits a smooth structure whose universal cover is diffeomorphic to $\R^n$. Proving these submanifolds are Hadamard, and classifing the uniformly negative ones, would give an even stronger constraint.

	So far, the achievements in this direction are only partial: by \cite[Theorem~B]{sst}, maximal hypersurfaces with regular boundaries are negatively curved up to a compact subset, and the curvature approximates exponentially fast $-1$ over diverging sequences; \cite{beykas} classifies, among the equivariant hypersurfaces, the Gromov-hyperbolic ones, namely the manifolds which resemble negatively curved ones from a metric point of view, and their work shows that the \textit{not} Gromov-hyperbolic ones correspond exactly to the ones maximizing the width. This class of admissible boundaries are the object of the upcoming work \cite{quasisph}. Finally, \cite[Theorem~C]{mt} shows that maximal hypersurfaces in the Anti-de Sitter space have non-positive Ricci curvature. 
	
	Corollary~\ref{cor:sectional} gives a stronger estimate on the intrinsic geometry of CMC hypersurfaces, and shows an explicit dependence of the sectional curvature on the width of the $H-$shifted convex hull. However, we are not able to fully characterize the sectional curvature of CMC hypersurfaces in terms of their asymptotic boundary.
	
	\subsubsection{Higher higher Teichm\"uller theory} The embedding of $\tilde{S}\times\R$ inside $\hyp^{n,1}$ induces a discrete and faithful representation
	\[\rho\colon\pi_1(S)\to\mathrm{Isom}(\hyp^{n,1})=\mathrm{O}(n,2),\]
	Higher higher Teichm\"uller theory is the branch of mathematics devoted to the study of connected component of the $G-$character variety of the fundamental group of a closed $n-$manifold $S$, for $G$ a semi-simple Lie group of rank higher than 2, consisting entirely of discrete and faithful representations. 
	
	For $G=\mathrm{O}(n,2)$ (resp. $G=\mathrm{O}(n,k+1)$), higher higher Teichm\"uller theory studies the moduli space of Anti-de Sitter (resp. pseudo-hyperbolic) structure of $S\times\R$ (resp. $S\times\R^k$). For previous studies connecting pseudo-hyperbolic geometry with higher higher Teichm\"uller theory, we suggest \cite{barb,barmer,dgk,ctt,ltw,hnnhiggs,lt,holo,sst,beykas,mv} and we refer to \cite{wien} for a survey on higher higher Teichm\"uller theory, although it focuses on surface groups.
	
	\subsection{Lower bound} Extending \cite[Proposition~1.C]{andreamax}, we prove also the inequality in the opposite direction of Theorem~\ref{thm:Schauder}. Indeed, the width of the $H-$shifted convex hull of $\Lambda$ is bounded from above by the extrinsic curvature of the $H-$hypersurface $\Sigma$. 
	\begin{thmx}\label{pro:w<B}
		Let $\Lambda$ be an admissible boundary in $\pd_\infty\hyp^{n,1}$ and $H\in\R$. Let $\Sigma$ the unique properly embedded spacelike $H-$hypersurface such that $\pd_\infty\Sigma=\Lambda$. Then
		\[\omega_H(\Lambda)\le\arctan\left(\sup_\Sigma\lambda_1\right)-\arctan\left(\inf_\Sigma \lambda_n\right),\]
		for $\lambda_1\ge\dots\ge\lambda_n$ be the principal curvatures of $\Sigma$, decreasingly ordered.
	\end{thmx}
	When the traceless shape operator is small enough, which is always the case for $\omega_H(\pd_\infty\Sigma)$ small enough, thanks to Theorem~\ref{thm:Schauder}, the inequality of Theorem~\ref{pro:w<B} can be written in a more expressive way. 
	\begin{corx}\label{cor:w<B}
		Let $\Lambda$ be an admissible boundary and $H\in\R$. Let $B_0$ be the traceless shape operator of the properly embedded spacelike $H-$hypersurface such that $\pd_\infty\Sigma=\Lambda$. If $\|B_0\|_{C^0(\Sigma)}^2\le 1+(H/n)^2$, the width of $\ch_H(\Lambda)$ satisfies
		\[\tan\left(\omega_H(\Lambda)\right)\le\frac{2\|B_0\|_{C^0(\Sigma)}}{1+(H/n)^2-\|B_0\|_{C^0(\Sigma)}^2}.\]
	\end{corx}
	
	By Gauss equation, Corollary~\ref{cor:w<B} translates in terms of sectional curvature, in the $3-$dimensional case.
	\begin{corx}\label{lem:sect}
		Let $\Lambda$ be an admissible boundary in $\pd_\infty\hyp^{2,1}$ and $H\in\R$. Let $B_0$ be the traceless shape operator of the properly embedded spacelike $H-$hypersurface such that $\pd_\infty\Sigma=\Lambda$. Then,
		\[\tan\left(\omega_H(\Lambda)\right)\le\frac{2\|B_0\|_{C^0(\Sigma)}}{-\sup_\Sigma K_\Sigma}.\]
	\end{corx}
	We recall that, for CMC surfaces, $K_\Sigma\le0$, and $\sup_\Sigma K_\Sigma<0$ if and only if $\Lambda$ is the graph of a quasi-symmetric homeomorphism of the circle, in the splitting of $\pd_\infty\hyp^{2,1}$ as $\R\pj^1\times\R\pj^1$ we will describe in the next subsection.
	
	\subsection{Teichm\"uller theory}
	The \textit{universal Teichm\"uller space} is the space of \textit{quasi-symmetric} homeomorphism of $\sph^1$ (Definition~\ref{de:quasisym}), up to post composition by $\psl$. The well known Ahlfors-Beuring theorem connects quasi-symmetric homeomorphisms of the circle with quasiconformal map of the disc (Definition~\ref{de:quasiconf}):
	\begin{thm*}[\cite{ahlbeu}]
		Every quasiconformal map $\Phi\colon\mathbb{D}^2\to\mathbb{D}^2$ extends to a unique quasi-symmetric map $\phi\colon\sph^1\to\sph^1$. Conversely, any quasi-symmetric map $\phi\colon\sph^1\to\sph^1$ admits a quasiconformal extension $\Phi\colon\mathbb{D}^2\to\mathbb{D}^2$.
	\end{thm*}
	The quasiconformal extension is far from being unique, and it is a classical topic in Teichm\"uller theory to construct a suitable class of quasiconformal extensions to study the universal Teichm\"uller space. A classical problem consists of the comparison between the cross-ratio norm (Definition~\ref{de:quasisym}) of the quasi-symmetric map $\phi$, denoted by $\|\phi\|_{cr}$, and the \textit{quasiconformal dilatation} (see Equation~\eqref{eq:maxdil}) of the quasiconformal map extension $\Phi\colon\hyp^2\to\hyp^2$ in the preferred class, denoted by $\mathcal{K}(\Phi)$. Classical results in this direction are the estimates contained in \cite{ahlbeu,ahlbeuconstant,douear,douearlconstant}.
	
	Anti-de Sitter geometry in dimension $3$ has played an important role in Teichm\"uller theory since the groudbreaking work \cite{mess}. For $n=2$ the asymptotic boundary is ruled by two families of lightlike lines. These two families identify $\pd_\infty\hyp^{2,1}$ to $\R\pj^1\times\R\pj^1$: through this splitting, admissible boundaries become graphs of orientation preserving homeomorphisms of the circle. Furthermore, spacelike surfaces in $\hyp^{2,1}$ induce diffeomorphisms of the hyperbolic plane $\hyp^2$, through the so called \textit{Gauss map} (see Section~\ref{sub:gauss}). It turns out that \textit{minimal Lagrangian} diffeomorphisms, which have been widely studied (see for example \cite{minSchoen,minLab}), are induced by maximal surfaces: this correspondence has been exploited in \cite{minimalAdS,particles,univ,tou,andreamax} to study such diffeomorphisms using Anti-de Sitter geometry.
	
	Minimal Lagrangian diffeomorphisms are a particular case of $\theta-$landslides (Definition~\ref{de:theta}), for $\theta=\pi/2$. $\theta-$landslides have been introduced in \cite{theta} as smooth versions of earthquakes: if $\phi$ is a quasi-symmetric map, then the $\theta-$landslides $\Phi_\theta$ extending $\phi$ interpolates between the left and the right earthquake extending $\phi$, as $\theta$ varies in $(0,\pi)$. By the work \cite{areapres}, we know that the diffeomorphism induced by a CMC surface $\Sigma$ in $\hyp^{2,1}$ is a $\theta-$landslide, for $\theta$ depending on the mean curvature of $\Sigma$.
	
	In \cite[Theorem 2.A, Theorem 2.B and Corollary 2.D]{andreamax}, the quasiconformal dilatation of a quasiconformal minimal Lagrangian map $\Phi_{\pi/2}$ is bounded by the width of the convex hull of the admissible boundary of the maximal surface corresponding to $\Phi_{\pi/2}$ given by the graph of $\Phi_{\pi/2}|_{\pd_\infty\hyp^2}$, which is an orientation preserving homeomorphism of $\R\pj^1=\pd_\infty\hyp^2$. An application of Theorem~\ref{thm:Schauder} is to extend \cite[Theorem 2.A]{andreamax} to $\Phi_\theta$, for any $\theta$.
	\begin{thmx}\label{thm:qconformal}
		For any $\alpha\in(0,\pi/2)$, there exists universal constants $Q_\alpha,\eta_\alpha>0$ such that for all $\theta\in[\alpha,\pi-\alpha]$ and $\phi$ quasi-symmetric map satifing $\|\phi\|_{cr}\le\eta_\alpha$, then
		\[
		\ln(\mathcal{K}(\Phi_\theta))\le Q_\alpha\|\phi\|_{cr},
		\]
		for $\Phi_\theta$ the unique $\theta-$landslide extending $\phi$.
	\end{thmx}
	
	\subsection{Sketch of the proofs}\label{sec:sketch}
	The core of this paper is Theorem~\ref{thm:Schauder}: to prove it, let us fix a properly embedded $H-$hypersurface $\Sigma$, a point $x\in\Sigma$ and a totally umbilical $H-$hypersurface $\mathcal{P}_{\delta_H}$ disjoint from $\Sigma$. We study the function $v_H\colon\Sigma\to\R$ associating to each point $x\in\Sigma$ its distance from $\mathcal{P}_{\delta_H}$, since its second derivatives approximate the traceless shape operator of $\Sigma$ around $x$ (Corollary~\ref{cor:hess}). The function $v_H$ is bounded from above by the width $H-$shifted convex hull of $\Sigma$ (Proposition~\ref{pro:v<w}), hence the proof of Theorem~\ref{thm:Schauder} reduces to bound the $C^2-$norm of $v_H$ by its $C^0-$norm, up to a universal constant.
	
	In Proposition~\ref{pro:hess}, we show that $v_H$ satisfies the elliptic PDE
	\begin{equation}\label{eq:pde}
		\Delta_\Sigma v_H-nv_h=f_H,
	\end{equation}
	for $\Delta_\Sigma$ the Laplace-Beltrami operator on $\Sigma$ and $f_H$ an explicit real function essentially depending on $H$, and we bound $f_H$ from above by $v_H$, up to a constant. By Schauder estimates (Proposition~\ref{pro:pde}), the $C^2-$norm of $v_H$ with the $C^0-$norm of $v_H$ and $f_H$, up to a constant depending on the elliptic operator $L=\Delta_\Sigma-n$.
	
	The goal is almost achieved: so far we get that
	\[\|B_0(x)\|\le C\omega_H(\Lambda),\]
	but, \textit{a priori}, the multiplicative constant is far from universal: the technical part lies in proving that the procedure does not depend on the choice of the $H-$hypersurface $\Sigma$, the point $x$ and the umbilical hypersurface $\mathcal{P}_{\delta_H}$. Under necessary but not restrictive assumptions on $\mathcal{P}_{\delta_H}$ (Definition~\ref{de:CMCLP}), we give bounds for the gradient (Proposition~\ref{pro:gradbound}) and for the Hessian (Corollary~\ref{cor:hess}) of $v_H$ over an open ball around $x$ of fixed radius (Corollary~\ref{cor:metricschaud}), which not depend on the choice of $\Sigma$ and $x$.
	
	Finally, we prove that the Laplace-Beltrami operators are uniformly elliptic over the space of $H-$hypersurfaces (Lemma~\ref{lem:beltrami}), hence Schauder estimates do not depend on the choice of $\Sigma$, concluding the proof of Theorem~\ref{thm:Schauder}.
	
	As anticipated in Remark~\ref{rem:shauder}, a key role in finding this uniform estimate is played by the fact that the norm of the shape operator is uniformly bounded for $H-$hypersurfaces in Anti-de Sitter space.
	
	\subsubsection{Applications} To prove Corollary~\ref{cor:sectional}, we apply Theorem~\ref{thm:Schauder} to Gauss equation, relating the extrinsic curvature to the intrinsic one.
	
	The quasiconformal dilatation of a $\theta-$landslide is explicitly linked to norm of the traceless operator (Equation~\ref{eq:dilatation}), while the cross-ratio norm bounds from above the width (Lemma~\ref{lem:cross}). We conclude the proof of Theorem~\ref{thm:qconformal} by substituting the estimates of Theorem~\ref{thm:Schauder} in the aformentioned estimates. 
	
	\subsubsection{Escaping the $H-$shifted convex hull} Theorem~\ref{pro:w<B} has a very geometrical proof: to build an $H-$convex subset $\mathcal{C}_H$ modeled on the geometry of the $H-$hypersurface $\Sigma$, and in particular containing its asymptotic boundary $\Lambda$. By minimality, $\mathcal{C}_H$ contains the $H-$shifted convex hull $\ch_H(\Lambda)$, hence $\omega_H(\Lambda)$ is bounded from above by the width of $\mathcal{C}_H$. To achieve this goal, we study the shape operator of $H-$convex hypersurfaces (Lemma~\ref{lem:Hconvex}). The shape operator has an explicit description along the normal evolution (Corollary~\ref{cor:principalcurvature}). Hence, we can estimate the time when the leaves become future (resp. past) $H-$convex in terms of the shape operator of $\Sigma$. The portion of space delimited by the two leaves $\mathcal{C}_H$. The technical part of this proof is to study the degeneracy of the normal flow and its consequences.
		
	\subsection{Organization of the paper}
	The necessary background is presented in the Section~\ref{sec:ads}. In Section~\ref{sec:Hconvex}, we introduce the main objects of this work, \textit{i.e.} the notion of $H-$convexity and the $H-$shifted convex hull. In Section~\ref{sec:w<B}, we prove Theorem~\ref{pro:w<B} and deduce Corollary~\ref{cor:w<B}. 
	
	Section~\ref{sec:B<w} is the core and the most technical part of this paper, in which we prove Theorem~\ref{thm:Schauder}.
	The last two section are devoted to applications: we prove Corollary~\ref{cor:sectional} in Section~\ref{sec:sectional}, and Theorem~\ref{thm:qconformal} in Section~\ref{sec:teich}.
	
	\subsection*{Aknowledgments} I would like to thank Francesco Bonsante and Andrea Seppi for their tireless support and encouragement. I would also thank Thierry Barbot and Jérémy Toulisse who carefully read the first drafts of this paper, and helped me to improve it. Finally, I thank Giulia Cavalleri, Farid Diaf, Timothé Lemistre, Alex Moriani, Eleonora Maggiorelli, Nicholas Rungi, Rym Sma\"i, Graham Smith, Edoardo Tolotti and Gabriele Viaggi for related discussions.
		
	\section{Preliminaries}\label{sec:ads}
	A Lorentzian manifold $(M,g)$ is the data of a smooth manifold $M$ of dimension $n+1$ and a symmetric non-degenerate $(0,2)-$tensor $g$ of signature $(n,1)$. Tangent vectors are distinguished by their causal properties: a vector $v\in TM$ is \textit{timelike}, \textit{lightlike} or \textit{spacelike} if $g(v,v)$ is respectively negative, null or positive. Furthermore, non-spacelike vector are called \textit{causal}. A curve $c\colon I\to M$ is called timelike (resp. lightlike, spacelike, causal), if its tangent vector $c'(t)$ is timelike (resp. lightlike, spacelike, causal) for all $t\in I$. 
	
	\begin{de}
		Two points are \textit{time-related} (resp. \textit{light-related}, \textit{space-related}) if there exists a timelike (resp. lightlike, spacelike) geodesic joining them.
	\end{de}
	
	\begin{de}\label{de:space}
		A $C^1-$submanifold of $M$ is \textit{spacelike} if the induced metric is Riemannian.
	\end{de}
	
	A Lorentzian manifold is \textit{time-orientable} if the set of timelike vectors in $TM$ has two connected components, which will always be our case. A time-orientation is the choice of one of the connected component: vectors are called \textit{future-directed} if they belong to the chosen connected component, \textit{past-directed} otherwise. 
	
	\begin{de}\label{de:cone}
		The \textit{cone} of a subset $X$ of $M$ is the set $I(X)$ of points of $M$ that can be joined to $X$ by a timelike curve. 
		
		Once a time-orientation is set, one can distinguish the \textit{future} cone $I^+(X)$ and the \textit{past} cone $I^-(X)$, containing respectively the points that can be reached from $X$ along a future-directed or past-directed timelike curve.
	\end{de}
	\subsection{Anti-de Sitter geometry}
	Anti-de Sitter manifolds are the Lorentzian analogous of hyperbolic manifolds, \emph{i.e.} pseudo-Riemannian manifolds with signature $(n,1)$ and constant sectional curvature $-1$. In the following, we present two models of Anti-de Sitter geometry: the quadric model $\hyp^{n,1}$ and the universal cover $\hypu^{n,1}$.	
	
	\subsubsection{Quadric model}\label{sub:quadric}
	The pseudo-Euclidean space $\R^{n,2}$ is the vector space $\R^{n+2}$ endowed with the non-degenerate bilinear form of signature $(n,2)$
	\begin{equation*}
		\pr{x,y}:=x_1y_1+\dots+x_n y_n-x_{n+1}y_{n+1}-x_{n+2}y_{n+2}.
	\end{equation*}
	The quadric model for Anti-de Sitter geometry is the set of unitary negative vectors of $\R^{n,2}$, \emph{i.e.}
	\[\hyp^{n,1}:=\{x\in\R^{n,2}, \pr{x,x}=-1\},\]
	endowed with the iduced metric. It generalizes the hyperboloid model for the hyperbolic space in the Lorentzian setting.
	
	As for the hyperbolic space, $\hyp^{n,1}$ admits a conformal boundary, which consists of the oriented isotropic lines for the bilinear form $\pr{\cdot,\cdot}$, and it is conformal to the Einstein univers $\pd_\infty\hyp^n\times\sph^1$, endowed with the Lorentzian metric $g_{\sph^{n-1}}-g_{\sph^{1}}$. 
	
	The tangent space $T_x\hyp^{n,1}$ identifies with $x^\perp=\{y\in\R^{n,2},\pr{x,y}=0\}$ and the restriction of the scalar product to $T\hyp^{n,1}$ is a time-orientable Lorentzian metric with constant sectional curvature $-1$. 
	
	The isometry group of $\hyp^{n,1}$ is $\mathrm{O}(n,2)$, and the action \textit{maximal}: any linear isometry from $T_x\hyp^{n,1}$ to $T_y\hyp^{n,1}$ is the tangent map of an isometry of $\hyp^{n,1}$ sending $x$ to $y$. In particular, the action is transitive. 
	
	Totally geodesic $k-$submanifolds of $\hyp^{n,1}$ are precisely open subsets of the intersection between $\hyp^{n,1}$ and $(k+1)-$vector subspaces of $\R^{n,2}$. In the following, we will abusively refer to \textit{maximal} totally geodesic submanifold simply as totally geodesic submanifold, where maximality is to be intended in the sense of inclusion for connected submanifolds.
	
	In particular, geodesics of $\hyp^{n,1}$ starting from $x$ are of the form
	\begin{equation}\label{eq:geod}
		\exp_x(tv)=\begin{cases}
			\cos(t)x+\sin(t)v & \text{if $\pr{v,v}=-1$};\\
			x+tv & \text{if $\pr{v,v}=0$};\\
			\cosh(t)x+\sinh(t)v & \text{if $\pr{v,v}=1$}.
		\end{cases}
	\end{equation}
	Indeed, any curve $\gamma(t)$ satisfying one of the equations above is contained in $\hyp^{n,1}$, and an easy computation shows that $\gamma''(t)$, namely the covariant derivative of $\gamma$ with respect to the flat metric of $\R^{n,2}$, is proportional to $\gamma(t)$, \emph{i.e.} normal to $T_{\gamma(t)}\hyp^{n,1}$.

	\begin{rem}\label{oss:P+-}
		For $x\in\hyp^{n,1}$, the tangent space $T_x\hyp^{n,1}$ identifies with the linear subspace $x^\perp\sq\R^{n,2}$. The intersection of $x^\perp$ with $\hyp^{n,1}$ is the set of unitary timelike vectors in $T_x\hyp^{n,1}$. By the above discussion, $\hyp^{n,1}\cap x^\perp$ is a totally geodesic hypersurface, and a direct calculation shows that it is isometric to two disjoint copies of the hyperbolic space $\hyp^n$.
	\end{rem}
\begin{de}\label{de:P+-}
	We denote by $P_+(x)$ the set of \emph{future-directed} unitary timelike vectors and by $P_-(x)$ the set of \emph{past-directed} unitary timelike vectors in $T_x\hyp^{n,1}$, once a time-orientation is set.
\end{de} 
	Equation~\eqref{eq:geod} shows that timelike geodesics starting at $x$ are periodic curves $\gamma\colon\R\to\hyp^{n,1}$ such that $\gamma(k\pi)=(-1)^k x$. Moreover, as
	\[\exp_x\left(\pm\frac{\pi}{2}v\right)=\pm v,\]
	future-directed timelike geodesics intersect orthogonally $P_+(x)$ at $t=\pi/2+2k\pi$ and $P_-(x)$ at $t=-\pi/2+2k\pi$.

	\subsubsection{The universal cover}
	Let $P$ be a totally geodesic spacelike hypersurface of $\hyp^{n,1}$, and $p\in P$. Denote $N$ the unitary future-directed normal vector to $P$ at $p$, and define the map
	\begin{equation}\label{eq:split}
		\begin{tikzcd}[row sep=1ex]
			\psi_{(p,P)}\colon P\times\R\arrow[r] &\hyp^{n,1}\\
			(x,t)\arrow[r,maps to] & R_t(x),
		\end{tikzcd}
	\end{equation}
	where $R_t$ is the linear map which is a rotation of angle $t$ restricted to $\Span(p,N)$ and fixes its orthogonal complement $\Span(p,N)^\perp$.
	
	The map $\psi_{(p,P)}$ is a covering whose domain is simply connected for any pair $(p,P)$. Hence, the universal cover for the Anti-de Sitter space is $\hypu^{n,1}:=\hyp^n\times\R$, endowed with the pull-back metric
	\begin{equation*}
		\begin{split}
			g_{\hypu^{n,1}}:&=\psi_{(p,P)}^*g_{\hyp^{n,1}}=g_{P}-\pr{p,\cdot}^2dt^2\\
			&=g_{\hyp^n}-\cosh^2\left(\mathrm{d}_{\hyp^n}(p,\cdot)\right)dt^2,
		\end{split}
	\end{equation*}
	for $g_P$ the restriction of $g_{\hyp^{n,1}}$ to $P$, which is isometric to $\hyp^n$, while by Equation~\eqref{eq:geod} \[\cosh\left(\mathrm{d}_{\hyp^n}(p,\cdot)\right)=-\pr{p,\cdot}.\] 
	
	\begin{de}\label{de:split}
		A \textit{splitting} of $\hypu^{n,1}$ is the choice of a pair $(p,P)$ identifing $\hypu^{n,1}$ with $\hyp^{n}\times\R$. We denote $x_0:=(0,\dots,0,1)\in\hyp^n$, namely $\{x_0\}\times\R=\psi_{(p,P)}^{-1}(\gamma)$, for $\gamma$ the timelike geodesic normal to $P$ at $p$.
	\end{de}
	
	By lifting the isometries of $\hyp^{n,1}$, it turns out that $\hypu^{n,1}$ has maximal isometry group, too. Moreover, since $\psi_{(p,P)}$ restricted to the slices $P\times\{t\}$ is linear, $P\times\{t\}$ is a totally geodesic spacelike hypersurface, for all $t\in\R$. In contrast, $\{x_0\}\times\R$ is the only fiber which is a (timelike) geodesic. 
	
	We fix once and for all a time-orientation in the two models for Anti-de Sitter geometry: in the universal cover $\hypu^{n,1}$, we choose the one coinciding with the natural orientation of $\R$, and for the quadric model $\hyp^{n,1}$ the one induced by any covering map $\psi$ as in Equation~\eqref{eq:split}.
	
	\subsubsection{Timelike distance}\label{sub:Up}
	The length of a piecewise $C^1-$timelike curve $c\colon(a,b)\to\hypu^{n,1}$ is defined as
	\[
	\ell(c):=\int_a^b \sqrt{-g_{\hypu^{n,1}}\left(\dot{c}(t),\dot{c}(t)\right)}dt.
	\]
	\begin{de}\label{de:dist}
		Let $p\in\hypu^{n,1}$ and $q\in I(p)$. Their Lorentzian distance is
		\begin{equation*}
			\dist(p,q):=\sup\{\ell(c),c\text{ timelike curve joining $p$ and $q$}\},
		\end{equation*} 
	\end{de}
	The Lorentzian distance satisfies the reverse triangle inequality
	\begin{equation}\label{eq:triangleineq}
		\dist(p,q)\ge\dist(p,r)+\dist(r,q),
	\end{equation}
	provided that the three quantities are well defined and $r$ is chronologically between $p$ and $q$, that is either $r\in I^+(p)\cap I^-(q)$ or $r\in I^+(q)\cap I^-(p)$ (see Definition~\ref{de:cone}).
	
	For a splitting $(p,P)$, let $\gamma$ be the timelike geodesic in $\hyp^{n,1}$ parameterized by arclength so that $\gamma(0)=p$ and $\gamma'(0)\perp P$, namely 
	\[\gamma(t)=\psi_{(p,P)}\left(\{x_0\}\times\{t\}\right).\]
	Denote by $Q(t)$ the totally geodesic spacelike hyperplane orthogonal to $\gamma$ at $\gamma(t)$: a fundamental domain for $\psi_{(p,P)}$ is $\hyp^n\times\left(t,t+2\pi\right)$, for any $t\in\R$, and its image through $\psi_{(p,P)}$ is $\hyp^{n,1}\setminus Q(t)$.
	
	\begin{de}\label{de:Pp}
		For $p\in\hypu^{n,1}$, we denote $\mathcal{P}_+(p)$ (resp. $\mathcal{P}_-(p)$) the set at Lorentzian distance $\pi/2$ from $p$ in the future (resp. in the past). We will call it \textit{dual hypersurface} to $p$ in the future (resp. in the past) in $\hypu^{n,1}$. We also denote $p_\pm$ the unique point time-related to $p$ contained in $\{\dist(p,\cdot)=\pi\}\cap I^\pm(p)$.
	\end{de}
	
	\begin{rem}\label{oss:dist}
		The distance between $p,q$ is achieved through a timelike geodesic if $q\in I^-(p_+)\cap I^+(p_-)$ (\cite[Corollary~2.13, Lemma~2.14]{univ}). This condition is not restrictive for our purposes, as we will show in Remark~\ref{rem:invdom}.
	\end{rem}
	The timelike distance can be explicitly computed using Equation~\eqref{eq:geod}. We refer to \cite[Proposition~2.4.4]{ecrin} for the explicit computation.
	\begin{pro}\label{pro:distprod}
		For any splitting $\psi\colon\hypu^{n,1}\to\hyp^{n,1}$,
		\[
		\pr{\psi(p),\psi(q)}=-\cos\left(\dist(p,q)\right),
		\]
		for any pair $p,q$ of time-related points such that $q\in I^-(p_+)\cap I^+(p_-)$.
	\end{pro}
	
	\begin{cor}\label{cor:P}
		In any splitting, $\psi(\mathcal{P}_\pm(p))=P_\pm(\psi(p))$ and $\psi(p_\pm)=-\psi(p)$. In particular, $\mathcal{P}_\pm(p)$ is a totally geodesic spacelike hypersurface.
	\end{cor}
	
	\subsection{Spacelike hypersurfaces}\label{sec:graphs}
	From a qualitative point of view, properly embedded spacelike CMC hypersurfaces are well understood: properly embedded spacelike hypersurfaces in $\hypu^{n,1}$ are all graphs of functions $\hyp^n\to\R$ (Lemma~\ref{lem:graph}), the convergence of sequences of CMC hypersurfaces is fully determined by the convergence of their boundaries and mean curvature (Proposition~\ref{pro:compact}) and the strong maximum principle determines the reciprocal position of pairs of CMC hypersurfaces (Proposition~\ref{pro:maxprin}).
	
	\subsubsection{Graphs}
	A subset $X$ of $\hypu^{n,1}\cup\pd_\infty\hypu^{n,1}$ is \textit{achronal} (resp. \textit{acausal}) if no pair of points of $X$ can be joined by a timelike (resp. causal) curve of $\hypu^{n,1}\cup\pd_\infty\hypu^{n,1}$.
	
	An achronal subset meets each fiber $\{x\}\times\R$, for $x\in\hyp^n$, at most once: indeed, the fibers are timelike curves in $\hyp^{n,1}$. Hence, any achronal subset can be written as the graph of a function defined over a subset of $\hyp^n$. If the function is defined over the whole $\hyp^n$, we call the graph \textit{entire}.
	
	We recall here \cite[Lemma~4.1.2]{bonsep} and \cite[Lemma~4.1.3]{bonsep}, which are proved for $\hyp^{2,1}$. However, the arguments do not depend on the dimension.
	
	\begin{lem}\label{lem:graph}
		For a subset $X$ of $\hypu^{n,1}\cup\pd_\infty\hypu^{n,1}$, the following statements are equivalent:
		\begin{enumerate}
			\item $X$ is achronal (resp. acausal);\label{it:1}
			\item there exists a splitting $(p,P)$ such that $X$ is the graph of a $1-$Lipschitz (resp. strictly $1-$Lipschitz) function;\label{it:2}
			\item for any splitting $(p,P)$, $X$ is the graph of a $1-$Lipschitz (resp. strictly $1-$Lipschitz) function.\label{it:3}
		\end{enumerate}
		Moreover, an achronal hypersurface $\Sigma$ in $\hypu^{n,1}$ is properly embedded if and only if it is an entire graph.
	\end{lem}
	\begin{rem}
		The function is $1-$Lipschitz as a map from $\hyp^n\cup\pd_\infty\hyp^n$ endowed with the spherical metric induced by identifying it with a closed hemisphere of $\sph^n$.
	\end{rem}
	In particular, an entire achronal hypersurface of $\hypu^{n,1}$ has a unique $1-$Lipschitz extension to the asymptotic boundary (\cite[Theorem~1]{mcs}).
	
	Conversely, any $1-$Lipschitz function \[\phi\colon\pd_\infty\hyp^n=\sph^{n-1}\to\R\] can be extended to an achronal entire graph. It turns out that the extension is unique if and only if $\phi$ contains antipodal points, in which case the graph of the unique extension is a totally geodesic degenerate hyperplane (see \cite[Proposition~4.1.5]{ecrin}).
	\begin{de}\label{de:adm}
		A set $\Lambda\sq\pd_\infty\hypu^{n,1}$ is an \textit{admissible boundary} if it is the graph of a $1-$Lipschitz map $f\colon\pd_\infty\hyp^n\to\R$ containing no antipodal points.
	\end{de}	
	
	\subsubsection{Spacelike graphs}\label{sub:graph}
	As anticipated in Definition~\ref{de:space}, a $C^1-$embedded hypersurface $\Sigma$ is spacelike if the induced metric is Riemannian, or equivalently if the normal vector is timelike at every point.
	
	\begin{pro}{\cite[Lemma~4.1.5]{bonsep}}\label{pro:spaceacausal}
		A properly embedded spacelike hypersurface $\Sigma$ in $\hypu^{n,1}$ is acausal.
	\end{pro}
	
	\begin{rem}\label{oss:Up}
		Remark that subset $X$ is achronal if and only if it is contained in the complement of $I(p)$, for any $p\in X$. One can check that $\hypu^{n,1}\setminus I(p)$ is contained in the region bounded by $\mathcal{P}_-(p)$ and $\mathcal{P}_-(p)$ (see Definition~\ref{de:Pp}), hence in a fundamental domain of $\psi_{(p,P)}$. 
		
		It follows by Proposition~\ref{pro:spaceacausal} and Remark~\ref{oss:Up} that any properly embedded spacelike hypersurface is contained in a fundamental domain of for $\psi_{(p,P)}$, for a suitable choice of $p\in\hypu^{n,1}$. In other words, there is a natural correspondence between properly embedded spacelike hypersurfaces in the three models of Anti-de Sitter geometry introduced in this thesis.
	\end{rem}
	
	\subsubsection{CMC graphs} If $\Sigma$ is a $C^2-$spacelike hypersurface, the second fundamental form $\sff$ is defined as the projection of the ambient Levi-Civita connection over the normal bundle of $\Sigma$. Since $\Sigma$ has codimension one, the normal space $N\Sigma$ naturally identifies with $\R$ by choosing the future directed normal vector as generator, so we consider $\sff$ as a symmetric $(0,2)-$tensor on $\Sigma$.
	
	\begin{de}
		A $C^2-$spacelike hypersurface is a \textit{constant mean curvature} (CMC) hypersurface its mean curvature is constant.
		
		Hereafter, we will call $H-$hypersurface a CMC spacelike hypersurface with constant mean curvature $H$, when the value of its mean curvature is relevant.
	\end{de}
	
	We collect here two important result which will be repeatedly used in the rest of the article. Their prove can be found in \cite[Proposition~5.2.1]{ecrin} and \cite[Proposition~5.2.1]{ecrin}, respectively.
	
	\begin{pro}\label{pro:compact}
			Let $\Sigma_k$ be a sequence of properly embedded spacelike $H_k-$hypersurfaces, contained in a precompact set of $\hypu^{n,1}\cup\pd_\infty\hypu^{n,1}$. Up to taking a subsequence, we can assume that $\Sigma_k$ converges to an entire achronal graph $\Sigma_\infty$ and $H_k$ converges to $H_\infty\in\R\cup\{\pm\infty\}$. 
			
			Then, exactly one of the following holds:
			\begin{enumerate}
					\item if $\pd_\infty\Sigma_\infty$ is not admissible, then $\Sigma_\infty$ is a totally geodesic lightlike hypersurface;
					\item if $\pd_\infty\Sigma_\infty$ is admissible and $H_\infty=\pm\infty$, then $\Sigma_\infty=\pd_\mp(\Omega(\pd_\infty\Sigma_\infty))$;
					\item if $\pd_\infty\Sigma_\infty$ is admissible and $H_\infty\in\R$, then $\Sigma_\infty$ is a $H_\infty-$hypersurface.
				\end{enumerate}
			Moreover, in the last case, $\Sigma_k$ converges to $\Sigma_\infty$ smoothly as a graph over compact sets, in any splitting.
		\end{pro}

	\begin{pro}[Strong maximum principle for CMC hypersurfaces]\label{pro:maxprin}
			Let $\Sigma_i$ be a properly embedded spacelike $H_i-$hypersurface, for $i=1,2$. 
			
			If $H_1\ge H_2$ and $\pd_\infty\Sigma_1\sq\overline{I^-(\Sigma_2)}$, then either $\Sigma_1\sq I^-(\Sigma_2)$ or $\Sigma_1\sq\Sigma_2$.
		\end{pro}

	\subsection{Convexity}\label{sec:causal} A subset $\mathcal{C}$ of $\hypu^{n,1}\cup\pd_\infty\hypu^{n,1}$ is \textit{geodesically convex} if any pair of points in $\mathcal{C}$ is joined by at least one geodesic of $\hypu^{n,1}$ and any geodesic connecting them lies in $\mathcal{C}$. It follows that the intersection of geodesically convex sets is still geodesically convex.
	
	In general, the notion of geodesical convexity can be tricky. However, for our purpose, it coincides with the usual notion of convexity, after projectivization.
	\begin{lem}\label{lem:convproj} 
		An open subset $\mathcal{C}$ of $\hypu^{n,1}$ is geodesically convex if and only if it maps homeomorphically onto $\pj\circ\psi(\mathcal{C})$, for any $\phi$ as in Equation~\eqref{eq:split}, and its image is convex in $\R\pj^{n+1}$.
	\end{lem}
	For a more detailed discussion, we suggest \cite[Section~4.6]{bonsep}. 
	
	Motivated by this result, hereafter we will rather call \textit{convex} a geodesically convex subset of $\hypu^{n,1}\cup\pd_\infty\hypu^{n,1}$. Lemma~\ref{lem:convproj} extends the notion of geodesical convexity to the quadric model: an open subset $C$ of $\hyp^{n,1}\cup\pd_\infty\hyp^{n,1}$ is \textit{convex} if, equivalently
	\begin{enumerate}
		\item its cone in $\R^{n,2}$ is convex;
		\item its projectivization $\pj(C)$ is convex in $\R\pj^{n+1}$;
		\item $\psi^{-1}\left(C\right)$ is a union of geodesically convex subset of $\hyp^{n,1}\cup\pd_\infty\hyp^{n,1}$.
	\end{enumerate} 
	
	\subsubsection{Invisible domain} Let $X$ be an achronal subset of $\hypu^{n,1}\cup\pd_\infty\hypu^{n,1}$. The \textit{invisible domain} of $X$, denoted by $\Omega(X)$, is the set of points of $\hypu^{n,1}$ that are connected to $X$ by no causal curve.
	
	Equivalently, the set $\overline{\Omega(X)}$ is the union of all achronal subsets containing $X$. In light of Lemma~\ref{lem:graph}, let $X=\gr u$ in a splitting $(p,P)$: we can equivalently say that a point $q=(y,t)$ belongs to $\overline{\Omega(X)}$ if and only if there exists a $1-$Lipschitz function $\tilde{u}$ extending $u$ such that $\tilde{u}(y)=t$.
	
	Let us define $u^-$ (resp. $u^+$) as the infimum (resp. the supremum) of the $1-$Lipschitz functions extending $u$. These function are well defined and they are called \textit{extremal extension} of $u$, since by definition
	\[
	u^-\le\tilde{u}\le u^+.
	\]
	for any $1-$Lipschitz extension $\tilde{u}$ of $u$. Their graphs do not depend on the splitting: indeed, for two functions $f,g\colon\hyp^n\to\R$,
	\[
	f\le g \iff\gr f\sq I^-(\gr g),
	\]
	which does not depend on the splitting. These functions allows gives a more operative description of the invisible domain, as proved in \cite[Lemma~4.2.2]{bonsep}.
	
	\begin{lem}\label{lem:extremal}
		Let $X$ be an achronal subset of $\hypu^{n,1}\cup\pd_\infty\hypu^{n,1}$, $u^\pm$ its extremal extensions.
		\begin{enumerate}
			\item $\Omega(X)=I^+(\gr u^-)\cap I^-(\gr u^+)$;
			\item $\gr u^\pm$ is an achronal entire graph.
		\end{enumerate}
	\end{lem}

	By construction, if $X=\bigcup_{i\in I} X_i$ is an achronal subset, then \[\Omega(X)=\bigcap_{i\in I}\Omega(X_i).\]
	The invisible domain of a point $p\in\pd_\infty\hypu^{n,1}$ is the portion of space contained between the two totally geodesic degenerate hyperplane containing $p$: it follows that the invisible domain of an admissible boundary $\Lambda$ isometrically embeds in $\hyp^{n,1}$, via any projection $\psi$ as in Equation~\eqref{eq:split}, and it has an explicit description, which can be checked in light of Proposition~\ref{pro:distprod}:
	\begin{equation}\label{eq:convexinvdom}
		\psi\left(\Omega(\Lambda)\right)=\{x\in\hyp^{n,1},\pr{x,\psi(q)}<0,\,\forall q\in\Lambda\}.
	\end{equation}
	By Lemma~\ref{lem:convproj}, Equation~\eqref{eq:convexinvdom} shows that $\Omega(\Lambda)$ is a convex subset, for any admissible boundary $\Lambda$ (compare with \cite[Proposition~4.6.1]{bonsep}).
	
	\begin{rem}\label{rem:invdom}
	It follows that the invisible domain is well defined in the quadric model, and we will abusively denote it as $\Omega(\Lambda)$, for $\Lambda$ and admissible boundary in $\pd_\infty\hyp^{n,1}$.
	
	In fact, one can prove that $\overline{\Omega(\Lambda)}$ is convex, and the covering $\psi$ (Equation~\eqref{eq:split}) restricts to an homeomorphism onto its image (see \cite[Lemma~4.7]{univ} and \cite[Proposition~4.6.1]{bonsep}). 
	\end{rem}

	\subsubsection{Convex hull}\label{sub:ch} For a subset $X$ of $\hypu^{n,1}\cup\pd_\infty\hypu^{n,1}$ contained in a geodesically convex set $\mathcal{C}$, the \textit{convex hull} of $X$, denoted $\ch(X)$, is the smallest geodesically convex subset of $\hypu^{n,1}\cup\pd_\infty\hypu^{n,1}$ containing $X$.
	
	If $\Lambda$ an admissible boundary, by Remark~\ref{rem:invdom} the closure of its invisible domain is convex, and it does not depend on the model. It follows that the convex hull of an admissible boundary is well defined, and it does not depend on the model.
	
	Since $\overline{\Omega(\Lambda)}$ intersects the asymptotic boundary $\pd_\infty\hypu^{n,1}$ exactly in $\Lambda$, the asymptotic boundary of $\ch(\Lambda)$ coincides with $\Lambda$. Indeed, by minimality $\overline{\Omega(\Lambda)}$ contains $\ch(\Lambda)$, hence
	\[
	\Lambda\sq\ch(\Lambda)\cap\pd_\infty\hypu^{n,1}\sq\overline{\Omega(\Lambda)}\cap\pd_\infty\hypu^{n,1}=\Lambda.
	\]
	
	Moreover, the projection $\psi(\ch(X))\sq\hyp^{n,1}$ is the intersection of $\hyp^{n,1}$ and the convex hull of $X$ in $\R^{n,2}$. For a more detailed discussion, we suggest \cite[Section~4.6]{bonsep}.
	
	\subsubsection{Time functions}\label{sub:pastfut}
	The invisible domain of an admissible boundary $\Lambda$
	We introduce two relevant functions on $\Omega(\Lambda)$ and state some of their main properties.
	\begin{de}
		For $p\in\Omega(\Lambda)$, we denote $\tau_\past(p)$ the Lorentzian distance of $p$ from $\pd_-\Omega(\Lambda)$, that is
		\[\tau_\past(p):=\sup_{q\in\pd_-\Omega(\Lambda)\cap I^-(q)}\dist(p,q).\]
		Analogously, $\tau_\fut$ stands for the Lorentzian distance from $\pd_+\Omega(\Lambda)$.
	\end{de}
	These functions \textit{time functions}, namely they strictly monotone along timelike paths. Usually, in the literature it is specified if they are strictly increasing or strictly decreasing functions, and the latter are called \textit{reverse} time functions.
	
	Moreover, $\tau_\past$ and $\tau_\fut$ have further remarkable properties, for which are known in the literature as cosmological times (see for example \cite{time}), when restricted respectively to the past and the future of an admissible boundary.
	\begin{de}\label{de:pastfutpart}
		For an admissible boundary $\Lambda$, we define its \textit{past part} and its \textit{future part} to be 
		\begin{equation}\label{eq:pastfutpart}
			\begin{split}
				\past(\Lambda)&:=I^-\left(\pd_+\ch(\Lambda)\right)\cap\Omega(\Lambda);\\
				\fut(\Lambda)&:=I^+\left(\pd_-\ch(\Lambda)\right)\cap\Omega(\Lambda).
			\end{split}
		\end{equation}
	\end{de}
	
	The following result has been proved in \cite[Proposition~6.19]{benbon} for the 3-dimensional case. However, the argument does not depend on the dimension, as already remarked in \cite{univ}.
	
	\begin{pro}\label{pro:benbon}	
		Let $\Lambda$ be an admissible boundary. Then $\tau_\past$ is a cosmological time for $\past(\Lambda)$, taking values in $(0,\pi/2)$. Specifically, for every point $p\in\past(\Lambda)$, there exist exactly two points $\rho^\past_-(p)\in\pd_-\Omega(\Lambda)$ and $\rho^\past_+(p)\in\pd_+\ch(\Lambda)$ such that:
		\begin{enumerate}
			\item $p$ belongs to the timelike segment joining $\rho^\past_-(p)$ and $\rho^\past_+(p)$;
			\item $\tau_\past(p)=\dist(\rho^\past_-(p),p)$;
			\item $\dist(\rho^\past_-(p),\rho^\past_+(p))=\pi/2$;
			\item $P(\rho^\past_\pm(p))$ is a support plane for $\past(\Lambda)$ passing through $\rho^\past_\mp(p)$;
			\item $\tau_\past$ is $C^1$ and $\nablah\tau_\past(p)$ is the unitary timelike tangent vector such that
			\[\exp_p\left(\tau_\past(p)\nablah\tau_\past(p)\right)=\rho^\past_-(p).\]
		\end{enumerate}
		for $\nablah\tau_\past$ the gradient of $\tau_\past$.
	\end{pro}
	\begin{rem}
		A symmetric result holds for $\tau_\fut$ in $\fut(\Lambda)$.
	\end{rem}
	
	\section{$H-$convexity}\label{sec:Hconvex}
	
	For an admissible boundary $\Lambda\sq\pd_\infty\hyp^{n,1}$, we adapt the notion of convex hull to any value of $H\in\R$: the \textit{$H-$shifted convex hull} $\ch_H(\Lambda)$ (Definition~\ref{de:Hhull}). Then, we study how it controls the geometry of the $H-$hypersurface asymptotic to $\Lambda$.
	
	\subsection{Definition and properties}
	We recalled in Section~\ref{sub:ch} that a convex subset $X$ in $\hyp^{n,1}$ can be written as intersection of half-spaces bounded by totally geodesic achronal hypersurfaces. To define \textit{$H-$convexity}, we replace half-spaces with connected components of the complement of totally umbilical spacelike hypersurfaces with mean curvature $H$. Hereafter, we will denote by $\mathcal{P}_\delta$ the totally umbilical spacelike hypersurfaces with constant mean curvature $H=n\tan(\delta)$. Equivalently, they are the hypersurfaces at oriented distance $-\delta$ from the spacelike totally geodesic hypersurface $\mathcal{P}_0$ with whom they share the same boundary. 
	
	\begin{de}\label{de:Hconv}
		A subset $X$ of $\hypu^{n,1}$ is \textit{future-$H-$convex} if it can be written as the intersection of the future of totally umbilical spacelike hypersurfaces of mean curvature $H$, namely if there exists $(\mathcal{P}^j_{\delta_H})_{j\in J}$, for $\delta_H:=\arctan(H/n)$, such that
		\[X=\bigcap_{j\in J}I^+(\mathcal{P}^j_{\delta_H}).\]
		Conversely, the subset $X$ is \textit{past-$H-$convex} if 
		\[X=\bigcap_{j\in J}I^-(\mathcal{P}^j_{\delta_H}).\]
		Finally, $X$ is \textit{$H-$convex} if $X=X_+\cap X_-$, for $X_+,X_-$ a future-$H-$convex and a past-$H-$convex subset, respectively.
	\end{de}
	
	\begin{rem}
		One could expect that, for $H=0$, the actual notion of convexity is recovered. However, this is not the case: a convex set is $0-$convex if and only if it admits only achronal support hyperplanes. For example, a geodesic timelike segment is convex in $\hyp^{n,1}$ but not $0-$convex. On the opposite direction, one can check that $\hypu^{n,1}$ is $0-$convex but not convex.
	\end{rem}
	
	\begin{rem}\label{rem:timeorientation}
		One can easily check that $H-$convexity is preserved by time-preserving isometries, while time-reversing isometries send future-$H-$convex subsets to past-$(-H)-$convex ones, and viceversa. For this reason, most of the following result will be stated and proved for future-$H-$convex subsets, without loss of generality.
	\end{rem}
	
	The following maximum principle for umbilical hypersurfaces will be useful for discovering properties of $H-$convexity.
	\begin{lem}\label{lem:maxumbilical}
		Consider two umbilical hypersurfaces $\mathcal{P}_{\delta_H}$ and $\mathcal{P}_{\delta_L}$ in $\hypu^{n,1}$. If $H>L$ and they meet tangentially at a point $p$, then $\mathcal{P}_{\delta_H}\setminus\{p\}\sq I^+(\mathcal{P}_{\delta_L})$.
	\end{lem}
	\begin{proof}
		For $h\in\R$, denote by $\delta_h=\arctan(h/n)$ and let $f_h\colon\hyp^n\to\R$ be the function defined by
		\begin{align*}
			f_h(x)&=\arccos\left(\frac{\sin(\delta_h)}{\sqrt{1+\sum_{i=1}^n x_i^2}}\right)-\arccos\left(\sin(\delta_h)\right).
		\end{align*}
		We claim that $\gr f_h$ is a $h-$umbilical spacelike hypersurface which is tangent to $\hyp^n\times\{0\}$ at $(x_0,0)$. In particular, choosing the splitting $(p,T_p\mathcal{P}_{\delta_H})=(p,T_p\mathcal{P}_{\delta_L})$, we have $\mathcal{P}_{\delta_H}=\gr f_H$ and $\mathcal{P}_{\delta_L}=\gr f_L$.
		
		One can check that the function $h\mapsto f_h(x)$ is strictly increasing for any $x\in\hyp^n$, $x\ne x_0$. Hence, $f_H(x)>f_L(x)$ for any $x\ne x_0$, that is \[\mathcal{P}_{\delta_H}\setminus\{p\}\sq I^+(\mathcal{P}{\delta_L}),\]
		which concludes the proof.
		
		To prove the claim, project the problem on $\hyp^{n,1}$ through $\psi$. One can check that $P_{\delta_h}:=\psi(\mathcal{P}_{\delta_h})$ is described by the equation
		\begin{equation}\label{eq:descrPdelta}
			\pr{p,e}=-\sin(\delta_h).
		\end{equation} 
		for a suitable choice of $e$. Indeed, in $\hyp^{n,1}$ the umbilical hypersurface $P_{\delta_h}$ is equidistant to a totally geodesic $P$. Let $e_\pm$ be the point such that $e_\pm^\perp=P$, then $\pr{p,e_\pm}=\mp\sin(\delta_h)$, that is $e:=e_-$ satisfies Equation~\eqref{eq:descrPdelta}.
		
		Timelike geodesics starting from $e$ intersect $P_{\delta_h}$ orthogonally: if $\mathcal{P}_{\delta_h}$ is tangent to $\hyp^n\times\{0\}$ at $(x_0,0)$, then $\psi^{-1}(e)$ lies in the fiber $\{x_0\}\times\R$. Denote by $e_h:=(x_0,T_h)$ the closest point to $\mathcal{P}_{\delta_h}$ in $\psi^{-1}(e)\cap I^-\left(\mathcal{P}_{\delta_h}\right)$, namely
		\[
		T_h=\arccos\left(\sin(\delta_h)\right)=\delta_h-\frac{\pi}{2}.
		\]
		
		Let $q=(x,t)\in\mathcal{P}_{\delta_h}$, for $x=(x_1,\dots,x_{n+1})\in\hyp^n$, then
		\begin{align*}\sin(\delta_h)&=-\pr{\psi(q),\psi(e_h)}=-\pr{\psi(q),e}=x_{n+1}\left(\cos(t)\cos(T_h)+\sin(t)\sin(T_h)\right)\\
			&=x_{n+1}\cos(t-T_h)=\sqrt{1+\sum_{i=1}^n x_i^2}\cos(t-T_h).\end{align*}
		Hence, in this splitting, the umbilical hypersurface $\mathcal{P}_{\delta_h}$ is the graph of the function
		\[f_h(x)=\arccos\left(\frac{\sin(\delta_h)}{\sqrt{1+\sum_{i=1}^n x_i^2}}\right)-T_h,\]
		proving the claim and concluding the proof.	
	\end{proof}
	
	\begin{lem}\label{lem:bordoconvesso}
		Let $X$ be a non-empty future-$H-$convex (resp. past-$H-$convex) set. Then, there exists a properly embedded achronal hypersurface $S$ such that $\overline{X}=\overline{I^+(S)}$ (resp. $\overline{X}=\overline{I^-(S)}$).
	\end{lem}	
	\begin{proof}
		Assume that $X$ is future-$H-$convex. In a splitting, the set
		\[\overline{X}=\bigcap_{j\in J}\overline{I^+(\mathcal{P}^j_{\delta_H})}\]
		is the epigraph of the function \[f=\sup_{j\in J}f_j,\] for $f_j$ the strictly $1-$Lipschitz functions describing the $\mathcal{P}^j_{\delta_H}$'s. It follows that the function $f$ is $1-$Lipschitz, \textit{i.e.} $S:=\gr f$ is an achronal properly embedded hypersurface by Lemma~\ref{lem:graph}.
	\end{proof}
	\begin{de}\label{de:convexsurf}
		With a slight abuse of notation, we will call $S$ a \textit{future-$H-$convex} (resp. \textit{past-$H-$convex}) hypersurface.
	\end{de}
	
	\begin{de}
		Let $S$ be a properly embedded achronal hypersurface, and $p\in S$. An achronal properly embedded hypersurface $\mathcal{P}$ is a past (resp. future) \textit{support} hypersurface to $S$ at $p$ if it contains $p$ and it is contained in $\overline{I^-(S)}$ (resp. $\overline{I^+(S)}$).
	\end{de}
	
	The following characterization descends directly from Lemma~\ref{lem:bordoconvesso}.
	\begin{cor}\label{cor:support}
		A properly embedded hypersurface $S$ is future-$H-$convex (resp. past-$H-$convex) if and only if it admits, at any point $p$, a past (resp. future) support hypersurface which is either an umbilical $H-$hypersurface or a totally geodesic degenerate hypersurface.
	\end{cor}
	\begin{proof}
		Let us fix $p\in S$. If $S$ is future-$H-$convex, there exists a collection of umbilical hypersurfaces $(\mathcal{P}_{\delta_H}^j)_{j\in J}$, such that $S\sq\overline{I^+(\mathcal{P}_{\delta_H}^j)}$ for all $j\in J$ and $S$ is the boundary of
		\[X=\bigcap_{j\in J}I^+(\mathcal{P}_{\delta_H}^j).\]
		Let us extract a sequence $(j_k)_{k\in\N}\sq J$ such that \[\dist(\mathcal{P}^{j_k}_{\delta_H},p)<1/k.\]
		Thanks to Proposition~\ref{pro:compact}, the sequence subconverges to a properly embedded achronal hypersurface $\mathcal{P}_\infty$ which is either a totally geodesic degenerate hypersurface or a $H-$hypersurface. In the latter case, it is clear that $\mathcal{P}_\infty$ is a totally umbilical spacelike hypersurface. By construction, the point $p$ belongs to $\mathcal{P}_{\infty}$ and by continuity $S\sq\overline{I^+(\mathcal{P}_\infty)}$, namely $\mathcal{P}_\infty$ is a support hypersuface for $S$, as requested. Since $p$ was arbitrary, this holds for any point of $S$.
		
		Conversely, let $\mathcal{P}_p$ be a support past hypersurface for $S$ at $p$ as in the statement: we can build a sequence $(P^k_p)_{k\in\N}$ of totally umbilical $H-$hypersurfaces lying in the past of $\mathcal{P}_p$ such that the limit in the Hausdorff topology of $\mathcal{P}_p^k$ is $\mathcal{P}_p$. If $\mathcal{P}_p$ is  totally umbilical spacelike $H-$hypersurface, take the splitting $(p,\mathcal{P})$, for $\mathcal{P}$ the totally geodesic spacelike hypersurface equidistant to $\mathcal{P}_p$, and define $\mathcal{P}_p^k$ to be the totally umbilical $H-$hypersurfaces equidistant to $\hyp^n\times\{-1/k\}$. Otherwise, that is if $\mathcal{P}_p$ is a totally geodesic degenerate hypersurface, take totally umbilical spacelike $H-$hypersurface equidistant to the totally geodesic spacelike hypersurface $\mathcal{P}_-(p_k)$, for $p_k$ a sequence in $\mathcal{P}_p$ converging to its dual future point (Definition~\ref{de:Pp}). Then, the set
		\[X:=\bigcap_{p\in S}\bigcap_{k\in\N}I^+\left(\mathcal{P}_p^k\right)\]
		is future $H-$convex and its boundary coincides with $S$, concluding the proof.	
		
	\end{proof}
		
	We can deduce that $H-$convexity is a closed property.
	\begin{cor}\label{cor:Hconvclosed}
		Let $X_k$ be a sequence of properly embedded future-$H-$convex subsets of $\hyp^{n,1}$, converging to $X_\infty$ in the Hausdorff topology. Then, $X_\infty$ is future-$H-$convex.
	\end{cor}
	\begin{proof}
		Denote $S_k$ the achronal entire hypersurfaces such that $\overline{X_k}=\overline{I^+(S_k)}$. In a splitting, the $S_k$'s are graphs of $1-$Lipschitz functions $f_k$, which are uniformly bounded since the sequence $X_k$ converges in the Hausdorff topology. It follows that the sequence $f_k$ converges uniformly to a $1-$Lipschitz function $f_\infty$, \textit{i.e.} the sequence $S_k$ converges to the entire achronal graph $S_\infty=\gr f_\infty$ in the Hausdorff topology, and $\overline{X_\infty}=\overline{I^+(S_\infty)}$.
		
		To prove that $S_\infty$ is future-$H-$convex, let $a_k:=\sup|f-f_k|$ and replace $f_k$ by $f_k-a_k$, so that $S_k$ lies in the past of $S_\infty$, for any $k\in\N$. Fix a point $p:=f_\infty(x)$ in $S_\infty$, and consider $p_k:=f_k(x)$. Since $S_k$ is future-$H-$convex, by Corollary~\ref{cor:support} there exists a past support hypersurface $\mathcal{P}_k$ to $S_k$ at $p_k$, which is a $H-$umbilical hypersurface or a degenerate totally geodesic hypersurface. In particular, we have
		\[S\sq\overline{I^+(S_k)}\sq\overline{I^+(\mathcal{P}_k)}.\]
		
		We claim that the sequence $\mathcal{P}_k$ converges to a past support hypersurface $\mathcal{P}_\infty$. Indeed, $\mathcal{P}_\infty$ contains $p$ and lies in the past of $S_\infty$ by continuity. Moreover, by Proposition~\ref{pro:compact}, $\mathcal{P}_\infty$ is either a $H-$umbilical hypersurface or a degenerate totally geodesic hypersurface. Since the choice of $p$ was arbitrary, by Corollary~\ref{cor:support} we conclude the proof.
	\end{proof}
	
	\begin{cor}\label{cor:Hcontinuity}
		Let $S$ be an achronal properly embedded hypersurface in $\hypu^{n,1}$, and $L\in\R$.
		\begin{itemize}
			\item If $S$ is future-$H-$convex, for every $H\in(-\infty,L)$, then $S$ is future-$L-$convex;
			\item if $S$ is past-$H-$convex, for every $H\in(L,+\infty)$, then $S$ is past-$L-$convex.
		\end{itemize}
	\end{cor}
	\begin{proof}
		The result follows directly from Corollary~\ref{cor:support}: fix a point $p\in S$, consider a sequence $H_k\nearrow L$.
		Since $S$ is future-$H_k-$convex, it admits at $p$ a past support hypersurface $\mathcal{P}_k$ which is a $H_k-$totally umbilical spacelike hypersurfaces or a totally geodesic degenerate hypersurface. By Proposition~\ref{pro:compact}, the sequence $\mathcal{P}_k$ converges to a support hypersurface $\mathcal{P}_\infty$ which is either a $L-$totally umbilical spacelike hypersurfaces or totally geodesic degenerate hypersurface, proving the $L-$convexity of $S$ at $p$. Since $p$ was arbitrary, this concludes the proof.
	\end{proof}
	
	\begin{cor}
		Let $S$ be an achronal properly embedded hypersurface in $\hypu^{n,1}$, and $H\ge L$.
		\begin{itemize}
			\item If $S$ is future-$H-$convex, then $S$ is future-$L-$convex;
			\item if $S$ is past-$L-$convex, then $S$ is past-$H-$convex.
		\end{itemize}
		In particular, if $H\ge0$ (resp. $H\le0$), then $X$ is future-convex (resp. past-convex).
	\end{cor}
	\begin{proof}
		Let $S$ be a future-$H-$convex hypersurface. To prove that $S$ future-$L-$convex as well, it suffices to exhibit, at each point $p\in S$, a past support $L-$umbilical hypersurface or a a past support totally geodesic degenerate hypersurface.
		
		Fix $p\in S$: by Corollary~\ref{cor:support}, there exists a past support $H-$umbilical hypersurface $\mathcal{P}_{\delta_H}$ or a totally geodesic degenerate hypersurface at $p$. In the latter case, we are done. Otherwise, by Lemma~\ref{lem:maxumbilical}, the hypersurface $\mathcal{P}_{\delta_L}$ tangent to $\mathcal{P}_{\delta_H}$ at $p$ is a past support hypersurface for $S$. Since $p$ was arbitrary, this concludes the proof.
	\end{proof}
	
	For regular hypersurfaces, the $H-$convexity is strictly linked with the definiteness of the shape operator.
	\begin{lem}\label{lem:Hconvex}
		An entire spacelike $C^2-$hypersurface in $\hyp^{n,1}$ is future-$H-$convex (resp. past-$H-$convex) if and only if its shape operator $B$ satisfies $B\ge(H/n)\Id$ (resp. $B\le(H/n)\Id$).
	\end{lem}
	\begin{proof}
		Let $S$ be an entire spacelike $C^2-$hypersurface. Pick a point $p\in S$, an umbilical hypersurface $\mathcal{P}_{\delta_H}$ tangent to $S$ at $p$. Choose a splitting $\hyp^{n,1}=\hyp^n\times\R$ such that $p=(x_0,0)$ and $T_p S=\hyp^n\times\{0\}$, and let $f,g\colon\hyp^n\to\R$ be the $1-$Lipschitz functions realizing the two entire spacelike hypersurface $S$ and $\mathcal{P}_{\delta_H}$ as graphs. By construction, the point $x_0$ is a critical point for both $f$ and $g$, and $f(x_0)=g(x_0)=0$.
		
		First, assume $S$ is future-$H-$convex: it follows that $\mathcal{P}_{\delta_H}$ is a past support umbilical hypersurface for $S$, namely $f\ge g$. In particular, the function $f-g$ is non-negative and it achieves a global minimum at $p=(x_0,0)$. Hence, at such a point \[\hess(f-g)=B-\frac{H}{n}\Id\ge0.\]
		
		To prove the converse implication, we assume that $B$ is \textit{strictly} greater than $(H/n)\Id$. The general case follows from Corollary~\ref{cor:Hcontinuity}: indeed, if $B\ge(H/n)\Id$, it follows that $B>(L/n)\Id$, hence $\Sigma$ is future-$L-$convex, for any $L<H$.
		
		Fix a $H-$umbilical hypersurface $\mathcal{P}_{\delta_H}$ tangent to $S$ at a point $p$: we need to prove that $\mathcal{P}_{\delta_H}$ is a past support hypersurface for $\Sigma$ at $p$. In fact, we are proving a strong maximum principle-like statement similar to Lemma~\ref{lem:maxumbilical}, namely that $S\setminus\{p\}\sq I^+(\mathcal{P}_{\delta_H})$.
		
		In the splitting $(p,T_p S)$, let $S=\gr f$ and $\mathcal{P}_{\delta_H}=\gr g$. Since $d_{x_0}f=d_{x_0}g$ and \[\hess_{x_0}(f-g)=B(x_0)-(H/n)\Id>0,\] the point $x_0$ is a \textit{strict} local minimum for the function $f-g$. Then, there exists an open neighbourhood $U$ of $p$ in $S$ such that $U\setminus\{p\}\sq I^+(\mathcal{P}_{\delta_H})$.
		
		By contradiction, assume there exists a point $q\in S\setminus\{p\}$ belonging to $\mathcal{P}_{\delta_H}$. Consider the totally geodesic Lorentzian plane $\mathcal{Q}$ containing the geodesic connecting $p$ and $q$ and the geodesic $\exp_p(\R N^S_p)$. Denote by $\gamma\colon[a,b]\to S$ the curve connecting $p$ and $q$ in $S\cap\mathcal{Q}$, parameterized by arc-length. We claim that $\gamma$ is contained in the invisible domain of $\pd_\infty\mathcal{P}_{\delta_H}$. Indeed, denote by $t_1$, $t_2$ the projection of $p,q$ to \[\hyp^1\times\{0\}=(\hyp^n\times\{0\})\cap\mathcal{Q}\] and define the function
		\begin{equation*}
			h(t)=\begin{cases}
				g(t) & \text{if $t\in[t_1,t_2]$}\\
				f(t) & \text{otherwise}.
			\end{cases}
		\end{equation*}
		Since both $p=\gamma(a)$ and $q=\gamma(b)$ lie on the $\mathcal{P}_{\delta_H}$, the function $h$ is continuous. It follows that $h$ is a (strictly) $1-$Lipschitz map which agrees with $f$ at the boundary, hence is contained in the invisible domain of $\pd_\infty\mathcal{P}_{\delta_H}\cap\mathcal{Q}$ in $\hyp^{1,1}$. Denote by $e$ the dual past point of $\mathcal{P}_0$ (see Definition~\ref{de:Pp}): since $e$ is contained in $\mathcal{Q}$, one can check that $\Omega\left(\pd_\infty\mathcal{P}_{\delta_H}\cap\mathcal{Q}\right)=\Omega\left(\pd_\infty\mathcal{P}_{\delta_H}\right)\cap\mathcal{Q}$, hence $\gamma$ is contained in $\Omega(\pd_\infty\mathcal{P}_{\delta_H})$. 
		
		It follows that there exists a point $r\in\gamma$ maximising the distance from $e$. Then, $\gamma$ is tangent to $\mathcal{P}_{\delta_L}$ at $r$, for
		\begin{equation}\label{eq:Hconvex}
			\delta_L=-\dist(e,r)<-\dist\left(e,\gamma(a)\right)=-\dist(e,p)=\delta_H.
		\end{equation}
		
		Denote by $c$ the curve $\mathcal{P}_{\delta_L}\cap\mathcal{Q}$, parameterized by arc-length. Up to translation, we can assume $\gamma(0)=c(0)=r$. Since both $\gamma$ and $c$ are plane curves, we have that $\nablah_{\gamma'}\gamma'$ and $\nablah_{c'}{c'}$ are proportional at $t=0$. In particular, they are orthogonal to $\mathcal{P}_{\delta_L}$, since $c$ is a geodesic for $\mathcal{P}_{\delta_L}$. Moreover, since $r$ maximizes the distance from $e$, $\gamma$ is contained in $\overline{I^-(\mathcal{P}_{\delta_L})}$. Seeing $\gamma$ and $c$ as hypersurfaces of $\hypu^{1,1}$, the first part of the proof implies that the shape operator of $\gamma$ is greater or equal than $L/n$, that is \[\pr{\nablah_{\gamma'}\gamma'-\nablah_{c'}{c'},N^{\mathcal{P}_{\delta_L}}_r}\ge0,\]
		for $N^{\mathcal{P}_{\delta_L}}(r)$ the unitary future-directed vector normal to $\mathcal{P}_{\delta_L}$.
		
		It follows that
		\begin{align*}
			\frac{L}{n}&=\pr{B_{\mathcal{P}_{\delta_L}}\left(c'(0)\right),c'(0)}=-\pr{\nablah_{c'}{c'},N^{\mathcal{P}_{\delta_L}}_r}\\
			&\ge-\pr{\nablah_{\gamma'}\gamma',N^{\mathcal{P}_{\delta_L}}_r}\ge-\pr{\nablah_{\gamma'}\gamma',N^\Sigma_r}=\frac{H}{n},
		\end{align*} 
		contradicting Equation~\eqref{eq:Hconvex}, since $\delta_H=\arctan(H/n)$.
	\end{proof}
	
	\subsection{$H-$shifted convex hull} We are finally ready to introduce one of the main objects of this work.
	\begin{de}\label{de:Hhull}
		Let $\Lambda$ be an admissible boundary in $\pd_\infty\hypu^{n,1}$ (resp. $\pd_\infty\hyp^{n,1}$). The \textit{$H-$shifted convex hull} of $\Lambda$, denoted by $\ch_H(\Lambda)$, is the smallest $H-$convex set of $\hypu^{n,1}\cup\pd_\infty\hypu^{n,1}$ (resp. $\hyp^{n,1}\cup\pd_\infty\hyp^{n,1}$) containing $\Lambda$.
	\end{de}
	
	The next proposition shows that the definition is well-posed and explicitly describes $\ch_H(\Lambda)$: indeed, the $H-$shifted convex hull of $\Lambda$ is the intersection of the future (resp. past) of $H-$umbilical hypersurfaces whose boundary is contained in the past (resp. future) of $\Lambda$. 
	\begin{pro}\label{pro:hconvex}
		Let $\Lambda\sq\pd_\infty\hypu^{n,1}$ be an admissible boundary and denote by $\delta_H=~\arctan(H/n)$, then
		\[\ch_H(\Lambda)=\bigcap_{\pd_\infty\mathcal{P}_{\delta_H}\sq I^-(\Lambda)}I^+\left(\mathcal{P}_{\delta_H}\right)\cap\bigcap_{\pd_\infty\mathcal{P}_{\delta_H}\sq I^+(\Lambda)}I^-\left(\mathcal{P}_{\delta_H}\right).\]	
	\end{pro}
	\begin{proof}
		Denote by $X$ the set claimed to coincide with $\ch_H(\Lambda)$. By construction, $X$ is $H-$convex and contains $\Lambda$, hence $\ch_H(\Lambda)\sq X$ by minimality. 
		
		Conversely, consider a point $p\notin\ch_H(\Lambda)$. Since $\ch_H(\Lambda)$ contains $\Lambda$, it is non-empty, hence it has two boundary components, $\pd_\pm\ch(\Lambda)$, which are properly embedded achronal hypersurfaces (see Lemma~\ref{lem:bordoconvesso}). Hence, without loss of generality, we can assume that $p$ lies in the past of $\pd_-\ch_H(\Lambda)$. Equivalently, there exists an umbilical hypersurface $\mathcal{P}_{\delta_H}$ such that
		\[
		\begin{cases}
			\ch_H(\Lambda)\sq I^+(\mathcal{P}_{\delta_H})\\
			p\in\overline{I^-(\mathcal{P}_{\delta_H})}.
		\end{cases}\]
		The first condition implies that $\Lambda\sq I^+(\pd_\infty\mathcal{P}_{\delta_H})$: then, by definition of $X$, the second condition states that $p\notin X$, that is $X\sq\ch_H(\Lambda)$, concluding the proof.
	\end{proof}
	
	\begin{lem}\label{lem:ch-contained-invisible}
		Let $\Lambda\sq\pd_\infty\hypu^{n,1}$ be an admissible boundary, then $\ch_H(\Lambda)$ is contained in $\overline{\Omega(\Lambda)}$.
	\end{lem}
	\begin{proof}
		By Lemma~\ref{lem:bordoconvesso}, the boundary components $\pd_\pm\ch_H(\Lambda)$ are properly embedded achronal hypersurfaces. One can easily check that their asymptotic boundary is $\Lambda$, hence they are contained in $\overline{\Omega(\Lambda)}$ by Lemma~\ref{lem:extremal}, concluding the proof.
	\end{proof}
	\begin{rem}
		Since the restriction of $\psi$ (see Equation~\eqref{eq:split}) to $\overline{\Omega(\Lambda)}$ is a diffeomorphism onto its image for any admissible boundary $\Lambda\sq\pd_\infty\hypu^{n,1}$ (Remark~\ref{rem:invdom}), we have 
		\[\ch_H\left(\psi(\Lambda)\right)=\psi\left(\ch_H(\Lambda)\right),\]
		namely the $H-$shifted convex hull does not depend on the model.
	\end{rem}
	
	In light of Lemma~\ref{lem:ch-contained-invisible}, we can characterize $\ch_H(\Lambda)$ in terms of the cosmological time functions $\tau_\past,\tau_\fut$, which have been defined in Equation~\eqref{eq:pastfutpart}.
	
	\begin{cor}\label{cor:ch-time}
		Let $\Lambda$ be an admissible boundary in $\hypu^{n,1}$. Then
		\[\ch_H(\Lambda)=\left\{p\in\overline{\Omega(\Lambda)}\ |\ \tau_\past(p)\le\frac{\pi}{2}-\delta_H,\ \tau_\fut(p)\ge\frac{\pi}{2}+\delta_H\right\},\]
		for $\delta_H=\arctan(H/n)$.
	\end{cor}
	\begin{proof}
		It suffices to prove that
		\begin{align*}
			\bigcap_{\pd_\infty\mathcal{P}_{\delta_H}\sq I^+(\Lambda)}I^-\left(\mathcal{P}_{\delta_H}\right)=\left\{\tau_\past\le\frac{\pi}{2}-\delta_H\right\}, && \bigcap_{\pd_\infty\mathcal{P}_{\delta_H}\sq I^-(\Lambda)}I^+\left(\mathcal{P}_{\delta_H}\right)=\left\{\tau_\fut\ge\frac{\pi}{2}+\delta_H\right\}.
		\end{align*}
		
		We claim the first identity holds: then, the second follows by applying a time-orientation reversing isometry, concluding the proof.
		
		To prove the claim, fix a point $p\in\overline{\Omega(\Lambda)}\cap I^+\left(\ch_H(\Lambda)\right)$. Then, there exists a future support hyperplane $\mathcal{P}_0$ for $\Lambda$ such that $p\in I^+\left(\mathcal{P}_{\delta_H}\right)$. By construction, the distance between $p$ and $P_0$ is strictly greater then $-\delta_H$, while the distance between $\pd_-\Omega(\Lambda)$ and a future support hyperplane is at least $\pi/2$: indeed $\tau_\past^{-1}(\pi/2)=\pd_+\ch(\Lambda)$ (Proposition~\ref{pro:benbon}).
		Hence, 
		\[\tau_\past(p)=\dist\left(p,\pd_-\Omega(\Lambda)\right)>\frac{\pi}{2}-\delta_H.\]
		
		Conversely, assume $\tau_\past(p)>(\pi/2)-\delta_H$ consider a timelike geodesic $\gamma$ realizing the distance $\tau_\past(p)$ from $\pd_-\Omega(\Lambda)$. Then, $\gamma$ intersects orthogonally $\pd_-\Omega(\Lambda)$ at the point $\gamma(0)$ and $p=\gamma\left(\tau_\past(p)\right)$. 
		
		By Proposition~\ref{pro:benbon}, the dual future hyperplane $\mathcal{P}_0=\gamma(0)_+^\perp$ is a future support hyperplane for $\pd_+\ch(\Lambda)$ containing $\gamma(\pi/2)$. It follows that $p=\gamma\left(\tau_\past(p)\right)$ is at distance $\tau_\past(p)-(\pi/2)$ from $\mathcal{P}_0$, or equivalently $p$ belongs to $\mathcal{P}_{(\pi/2)-\tau_\past(p)}$, which is the umbilical hypersurface at oriented distance $\tau_\past(p)-(\pi/2)$ from $P_0$.
		
		Since $\tau_\past(p)>(\pi/2)-\delta_H$, the umbilical hypersurface $\mathcal{P}_{(\pi/2)-\tau_\past(p)}$ lies in the future of $\mathcal{P}_{\delta_H}$, hence \[p\notin\bigcap_{\pd_\infty\mathcal{P}_{\delta_H}\sq I^-(\Lambda)}\overline{I^-\left(\mathcal{P}_{\delta_H}\right)},\]
		proving the claim and concluding the proof.
	\end{proof}
	\begin{rem}
		For $H\ge0$, the future boundary of the $H-$shifted convex hull coincides with a leaf of the past cosmological time. More precisely,
		\[\pd_+\ch_H(\Lambda)=\tau_\past^{-1}\left(\frac{\pi}{2}-\delta_H\right).\]
		Hence, $\pd_+\ch_H(\Lambda)$ is $C^{1,1}$ by Proposition~\ref{pro:benbon}. On the other hand, no regularity is assured for $\pd_-\ch_H(\Lambda)$, except for the fact that it is a $1-$Lipschitz graph, since it is acausal by Lemma~\ref{lem:bordoconvesso}.
		
		Conversely, for $H\le0$, the past boundary is a leaf of $\tau_\fut$.
	\end{rem}
	
	As a consequence of the strong maximum principle for CMC hypersurface (Proposition~\ref{pro:maxprin}), we show that $H-$hypersurfaces are contained in the $H-$shifted convex hull of their boundary.
	
	\begin{cor}\label{cor:Hmaxprin}
		Let $H\in\R$ and $\Sigma_H$ be a properly embedded spacelike $H-$hypersurface. If $\Sigma_H$ is totally umbilical, then $\Sigma_H=\ch_H(\pd_\infty\Sigma_H)$, otherwise $\Sigma_H\sq\mathrm{Int}\left(\ch_H(\pd_\infty\Sigma_H)\right)$.
	\end{cor}
	\begin{proof}
		As for the convex hull, one can easily check that the interior of $\ch_H(\Lambda)$ is empty if and only if $\Lambda$ is the boundary of a totally geodesic spacelike hypersurface, or equivalently the boundary of a totally umbilical CMC spacelike hypersurface.
		
		The boundary of $I^\pm(\mathcal{P}_{\delta_H})$ in $\hyp^{n,1}$ is an umbilical hypersurface with constant mean curvature $H$. By the strong maximum principle (Proposition~\ref{pro:maxprin}), if $\pd_\infty\mathcal{P}\sq\overline{I^\pm(\Lambda)}$, then $\Sigma\sq I^\pm(\mathcal{P}_{\delta_H})$, or $\Sigma=\mathcal{P}_{\delta_H}$ concluding the proof.
	\end{proof}
	
	\subsection{Width} We introduce the timelike diameter of $\ch_H(\Lambda)$ to measure how far an admissible boundary $\Lambda$ is from being a totally geodesic one.
	
	\begin{de}\label{de:width}
		Let $A$ be a subset $\sq\hyp^{n,1}$, the \textit{width} of $A$ is the supremum of the length of timelike curves contained in $A$. Hereafter, $\omega_H(\Lambda)$ will denote the width of the $H-$shifted convex hull $\ch_H(\Lambda)$.
	\end{de}
	
	A first estimate for the width of $\ch_H(\Lambda)$ follows from Corollary~\ref{cor:ch-time}.
	\begin{cor}\label{cor:width}
		Let $\Lambda\sq\pd_\infty\hyp^{n,1}$ be an admissible boundary, then \[\omega_H(\Lambda)\le\frac{\pi}{2}-|\delta_H|,\] for $\omega_H(\Lambda)$ the width of $\ch_H(\Lambda)$ and $\delta_H=\arctan(H/n)$.
	\end{cor}
	\begin{proof}
		Without loss of generality, let us assume $H\ge0$. In this case, by Corollary~\ref{cor:ch-time} we have
		\[\ch_H(\Lambda)\sq\tau_\past^{-1}\left(\left[0,\frac{\pi}{2}-\delta_H\right]\right),\] 
		namely $\pd_+\ch_H(\Lambda)$ is the leaf of cosmological time on the past part of $\Lambda$. By definition of the cosmological time, the width of $\tau_\past^{-1}(I)$ coincides with the length of the interval $I$, which concludes the proof.
	\end{proof}
	
	\begin{lem}\label{lem:width}
		Let $\Lambda\sq\pd_\infty\hyp^{n,1}$ be an admissible boundary. Let $H,K\in\R$ such that $|H|\ge|K|$ and $HK\ge0$. Then
		\[
		\omega_K(\Lambda)+\delta_K-\delta_H\le\omega_H(\Lambda)\le\omega_K(\Lambda).
		\] 
	\end{lem}
	\begin{figure}[h]
		\centering
		\includegraphics{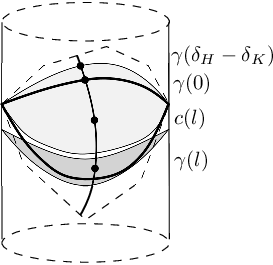}
		\caption{Proof of Lemma~\ref{lem:width}: $\past(\Lambda)$ is dashed, $\ch_K(\Lambda)$ in light gray, $\ch_H(\Lambda)$ in heavier line, and $U_{K,H}$ in dark gray.}
	\end{figure}
	\begin{proof}
		Let us fix $\Lambda$ and admissible boundary and, without loss of generality, assume $H>K\ge0$. To prove the first inequality, it suffices to remark that 
		\[\ch_H(\Lambda)\supseteq I^-\left(\pd_+\ch_H(\Lambda)\right)\cap I^+\left(\pd_-\ch_K(\Lambda)\right),\]
		which is $\ch_K(\Lambda)$ deprived of
		\[\overline{I^-\left(\pd_+\ch_K(\Lambda)\right)}\cap\overline{I^+\left(\pd_-\ch_H(\Lambda)\right)}=\left\{\frac{\pi}{2}-\delta_H\le\tau_\past\le\frac{\pi}{2}-\delta_K\right\},\]
		whose width is $\delta_H-\delta_K$.
		
		To prove the second inequality, consider a timelike curve $c$ of length $l$ contained in $\ch_H(\Lambda)$: we want to construct a curve with the same length in $\ch_K(\Lambda)$.
		
		Assume that $c$ is past-directed. Since it contained in $\ch_H(\Lambda)$, we have that 
		\[\dist\left(c(l),\pd_+\ch(\Lambda)\right)\ge l.\]
		Consider the integral line for $\nablah\tau_\past$ passing through $c(l)$, that is the timelike geodesic
		\[\gamma:=[\rho_+\left(c(l)\right),\rho_-\left(c(l)\right)]\]
		(compare with Proposition~\ref{pro:benbon}), parameterized so that it is past-directed and $\gamma(0)\in\pd_+\ch_H(\Lambda)$. Then, since 
		\[c(l)=\gamma\left(\dist\left(c(l),\pd_+\ch(\Lambda)\right)\right)\in I^+\left(\gamma(l)\right),\]
		we have that $\gamma([0,l])$ is contained in $\ch_H(\Lambda)$.
		
		We claim that the geodesic segment \[\gamma([\delta_K-\delta_H,l+\delta_K-\delta_H])\] is contained in $\ch_K(\Lambda)$. In particular, if we choose $\gamma$ of length $l=\omega_H(\Lambda)-\varepsilon$, we prove that $\omega_K(\Lambda)\ge\omega_H(\Lambda)-\varepsilon$: since $\varepsilon$ is arbitrary, we conclude that $\omega_K(\Lambda)\ge\omega_H(\Lambda)$.
		
		To prove the claim, we recall that $\gamma$ is an integral line for $\nablah\tau_\past$: hence, it meets orthogonally each level set of $\tau_\past$. By construction, the distance between two level sets is precisely the difference of the values they are preimages of, hence \[\gamma(\delta_K-\delta_H)\in\pd_+\ch_K(\Lambda).\] To conclude, we need to show that the point $p:=\gamma\left(l+\delta_K-\delta_H\right)$ is contained in $\ch_K(\Lambda)$. Consider an umbilical $K-$hypersurface $\mathcal{P}_{\delta_K}$ whose boundary lies in the past of $\Lambda$. By contradiction, assume that $p$ lies in the past of $\mathcal{P}_{\delta_K}$: hence, the geodesic segment 
		\[\gamma\left([l+\delta_K-\delta_H,l]\right)\] is contained in the open set 
		\[U_{K,H}:=I^-\left(\mathcal{P}_{\delta_K}\right)\cap I^+\left(\mathcal{P}_{\delta_H}\right),\] for the umbilical $H-$hypersurface $\mathcal{P}{\delta_H}$ sharing the same boundary as $\mathcal{P}_{\delta_K}$. By openness of $U_{K,H}$, the segment can prolonged, hence $U_{K,H}$ contains a timelike segment whose length is greater then $\delta_H-\delta_K$, which is the width of $U_{K,H}$, leading to a contradiction and concluding the proof.
	\end{proof}
	The following result descends directly.
	\begin{repprox}{pro:width}
		Let $\Lambda$ be an admissible boundary, the function $\omega_\bullet(\Lambda)\colon\R\to[0,\pi/2]$
		\begin{enumerate}
			\item is continuous;
			\item is increasing (resp. decreasing) for $H\le0$ (resp. for $H\ge0$);
			\item achieves its maximum at $H=0$;
			\item $\lim_{H\to\pm\infty}\omega_\bullet(\Lambda)=0$.
		\end{enumerate}
	\end{repprox}

\section{The width is a lower bound for the extrinsic curvature}\label{sec:w<B}
	In this section, we show that if a $H-$hypersurface is closed to be totally umbilical, \textit{i.e.} its traceless shape operator is small, then the $H-$width of its boundary is small. In particular, we prove
	\begin{repcorx}{cor:w<B}
		Let $\Lambda$ be an admissible boundary and $H\in\R$. Let $B_0$ be the traceless shape operator of the properly embedded spacelike $H-$hypersurface such that $\pd_\infty\Sigma=\Lambda$. If $\|B_0\|_{C^0(\Sigma)}^2\le 1+(H/n)^2$, the width of $\ch_H(\Lambda)$ satisfies
		\[\tan\left(\omega_H(\Lambda)\right)\le\frac{2\|B_0\|_{C^0(\Sigma)}}{1+(H/n)^2-\|B_0\|_{C^0(\Sigma)}^2}.\]
	\end{repcorx}
	In fact, we prove the following more general result, which does not require the additional bound on the norm of the traceless shape operator.
	\begin{repthmx}{pro:w<B}
		Let $\Lambda$ be an admissible boundary in $\pd_\infty\hyp^{n,1}$ and $H\in\R$. Let $\Sigma$ the unique properly embedded spacelike $H-$hypersurface such that $\pd_\infty\Sigma=\Lambda$. Then
		\[\omega_H(\Lambda)\le\arctan\left(\sup_\Sigma\lambda_1\right)-\arctan\left(\inf_\Sigma \lambda_n\right),\]
		for $\lambda_1\ge\dots\ge\lambda_n$ be the principal curvatures of $\Sigma$, decreasingly ordered.
	\end{repthmx}
	However, the estimate is more interesting near the Fuchsian locus, that is for small values of $\|B_0\|$, and the expression in Corollary~\ref{cor:w<B} is more direct than the one in Theorem~\ref{pro:w<B}.
	
	The plan is to use the normal evolution (Equation~\eqref{eq:normflow}) to compare the extrinsic curvature of an $H-$hypersurface $\Sigma$ to the width of the $H-$shifted convex hull of $\pd_\infty\Sigma$. In short, by Proposition~\ref{cor:Hmaxprin} $\Sigma$ is contained in $\ch_H(\pd_\infty\Sigma)$, and by Corollary~\ref{cor:principalcurvature} the principal curvatures are monotone along the normal flow. Hence, we look for the two times $T_+,T_-$ when $\Sigma_t$ becomes $H-$convex respectively in the past and in the future. The set $I^-(\Sigma_{T_+})\cap I^+(\Sigma_{T_-})$ is then $H-$convex and contains $\pd_\infty\Sigma$: by minimality, it contains the $H-$shifted convex hull, dominating its width.
	
	\subsection{Normal flow} Let $\Sigma$ be a $C^1-$spacelike hypersurface. Denote by $N$ the unit future-directed normal vector field over $\Sigma$. The \textit{normal flow} of $\Sigma$ is the function
	\begin{equation}\label{eq:normflow}
		\begin{tikzcd}[row sep=1ex]
			F\colon\Sigma\times\R\arrow[r] & \hypu^{n,1}\\
			(x,t)\arrow[r,maps to] & \exp_x(t N(x)).
		\end{tikzcd}
	\end{equation} 
	
	The pull-back metric on $\Sigma\times\R$ can be explicitly computed.
	\begin{lem}[Lemma~6.22 in \cite{bonsep}]\label{lem:normmetric}
		Let $\Sigma$ be a $C^2-$spacelike hypersurface in $\hypu^{n,1}$. Denote by $\Sigma_t:=F\left(\Sigma\times\{t\}\right)$ the leaves of the normal flow of $\Sigma$. The pull-back to metric on $\Sigma\times\R$ is
		\[(F^*g_{\hypu^{n,1}})_{(p,t)}(v,w)=-dt^2+g\left(\cos(t)v+\sin(t)B(v),\cos(t)w+\sin(t)B(w)\right),\]
		for $g$ the metric of $\Sigma$ and $B$ its shape operator.
	\end{lem}
	We can then compute the shape operator of the leaves of the normal flow, when they are non-degenerate.
	\begin{cor}\label{cor:principalcurvature}
		Let $\Sigma$ be a $C^2-$spacelike properly embedded hypersurface in $\hypu^{n,1}$. Assume that the leaf $\Sigma_t$ is non-degenerate. Then, the principal curvatures of $\Sigma_t$ are
		\[\lambda_i^t=\tan\left(\arctan(\lambda_i)-t\right),\]
		for $\lambda_i$ the the principal curvatures of $\Sigma$.
	\end{cor}
	
	\begin{lem}\label{pro:tempodeg}
		Let $\Sigma$ be a $C^2-$spacelike hypersurface. Denote by $\lambda_1\ge\dots\ge\lambda_n$ the principal curvatures of $\Sigma$. Define
		\begin{equation}\label{eq:Apm}
			A_+=\arctan\left(\inf_\Sigma\lambda_n\right)+\frac{\pi}{2} \qquad A_-=\arctan\left(\sup_\Sigma\lambda_1\right)-\frac{\pi}{2}	
		\end{equation}
		For any $t\in(A_-,A_+)$, the leaf of the normal flow at time $t$ is a spacelike hypersurface.
	\end{lem}
		
	\begin{proof}
		Denote by $g_t$ the induced metric, $F_t:=F|_{\Sigma\times\{t\}}$ and define
		\begin{align*}
			\beta(x,t)=\begin{cases}
				\cos(t)+\sin(t)\lambda_n(x) & \text{for $t\ge0$}\\
				\cos(t)+\sin(t)\lambda_1(x) & \text{for $t\le0$}
			\end{cases}.
		\end{align*}
		We claim that $(F_t)^*g_t\ge\beta(x,t)^2 g$, which concludes the proof since $\beta(\cdot,t)$ is a non-vanishing function over $\Sigma$, for all $t\in(A_-,A_+)$.
		
		To prove the claim, consider the orthonormal basis $(e_1,\dots,e_n)$ of $T_x\Sigma$, for $e_i$ the unit eigenvector relative to $\lambda_i(x)$. Take tangent vector \[v=\sum_{i=1}^n a_i e_i\in T_x\Sigma.\] By Lemma~\ref{lem:normmetric}, we have
		\begin{align*}
			(F_t)^*g_t(v,v)&=g\left(\cos(t)v+\sin(t)B(v),\cos(t)v+\sin(t)B(v)\right)\\
			&=\sum_{i=1}^n a_i^2 g\left(\cos(t)e_i+\sin(t)B(e_i),\cos(t)e_i+\sin(t)B(e_i)\right)\\
			&=\sum_{i=1}^n a_i^2\left(\cos(t)+\sin(t)\lambda_i(x)\right)^2.
		\end{align*}
		For $t>0$, denote by $a_+(x)\in(0,\pi)$ the solution of the equation
		\[-\frac{1}{\tan(s)}=\lambda_n(x)\iff s=\frac{\pi}{2}+\arctan\left(\lambda_n(x)\right).\]
		
		Since $\cos(t)+\sin(t)\lambda_i(x)>0$ for any $i=1,\dots,n$ and $t\in\left[0,a_+(x)\right)$, we get
		\[\left(\cos(t)+\sin(t)\lambda_i(x)\right)^2\ge\left(\cos(t)+\sin(t)\lambda_n(x)\right)^2=\beta(x,t)^2.
		\]
		By definition, $A_+=\inf_\Sigma a_+(x)$: hence, \[(F_t)^*g_t(v,v)\ge\beta(x,t)^2 g(v,v),\quad\forall v\in T\Sigma,\forall t\in[0,A_+),\] proving the claim for $t\ge0$. The same argument works for $t<0$, concluding the proof.
	\end{proof}

	\begin{rem}\label{rem:lunghezzanormflow}
		It is directly checked that $A_+ -A_-\le\pi$, with equality if and only if $\Sigma$ is totally umbilical.
	\end{rem}
		
	\subsection{The width is a lower bound for the extrinsic curvature}
	Lemma~\ref{lem:Hconvex} characterizes $H-$convexity in terms of principal curvatures, and Corollary~\ref{cor:principalcurvature} describes the principal curvature along the normal flow. Combining the two results, we can relate the extrinsic curvature of a $H-$hypersurface with the width of its $H-$shifted convex hull. 
	\begin{lem}\label{lem:tempoconvex}
		Let $\Sigma$ be a properly embedded spacelike hypersurface. Denote by $\lambda_1\ge\dots\ge\lambda_n$ the principal curvatures of $\Sigma$. Define
		\begin{align*}
			T_+:=\arctan\left(\sup_\Sigma\lambda_1\right), &&	T_-:=\arctan\left(\inf_\Sigma\lambda_n\right).
		\end{align*}
		Let $t\in(A_-,A_+)$ as in Lemma~\ref{pro:tempodeg}. Then, $\Sigma_t$ is past-$H-$convex (resp. future-$H-$convex) if and only if $t\ge T_+ -\delta_H$ (resp. $t\le T_- -\delta_H$).
	\end{lem}
	\begin{proof}
		Let us prove that $\Sigma_t$ is past-$H-$convex if and only if $t\ge T_+ -\delta_H$: the other proof is completely analogous.
		
		As remarked in Corollary~\ref{cor:principalcurvature}, the principal curvatures of $\Sigma_t$ are \[\lambda_i^t=\tan\left(\arctan(\lambda_i)-t\right).\]
		In particular, the function $t\mapsto\lambda_i^t$ is strictly decreasing in $t$. By Lemma~\ref{lem:Hconvex}, the leaf $\Sigma_t$ is past-$H-$convex if and only if $\lambda_i^t\le H/n$, for all $i=1,\dots,n$. This is equivalent to require
		\begin{align*}
			\frac{H}{n}&\ge\sup_\Sigma\lambda_1^t=\sup_\Sigma\tan\left(\arctan(\lambda_1)-t\right)=\tan\left(\arctan\left(\sup_\Sigma\lambda_1\right)-t\right)\\
			&\iff\delta_H\ge\arctan\left(\sup_\Sigma\lambda_1\right)-t=T_+-t,		
		\end{align*}
		concluding the proof.
	\end{proof}
	We can finally bound the width of the $\ch_H(\Lambda)$ using the extrinsic curvature of the unique $H-$hypersurface bounding $\Lambda$.
	\begin{repthmx}{pro:w<B}
		Let $\Lambda$ be an admissible boundary in $\hypu^{n,1}$ and $\Sigma$ the unique properly embedded spacelike $H-$hypersurface such that $\pd_\infty\Sigma=\Lambda$. Then
		\[\omega_H(\Lambda)\le\arctan\left(\sup_\Sigma\lambda_1\right)-\arctan\left(\inf_\Sigma\lambda_n\right).\]
	\end{repthmx}
	\begin{proof}
		Corollary~\ref{cor:Hmaxprin} implies that $\Sigma$ is contained in $\ch_H(\Lambda)$. We follow the normal flow of $\Sigma$ until the leaves become $H-$convex in the future (resp. in the past).
		
		Consider $A_\pm$ as in Equation~\eqref{eq:Apm} and $T_\pm$ as in Lemma~\ref{lem:tempoconvex}. In particular,
		\[A_+=T_- +\frac{\pi}{2},\qquad A_-=T_+ -\frac{\pi}{2}.\]
		
		We distinguish three different situation. First, assume that $T_+-\delta_H\ge A_+$. Then
		\begin{align*}
			T_+ -\delta_H\ge T_- +\frac{\pi}{2}\iff T_+-T_-\ge\frac{\pi}{2}+\delta_H\ge\frac{\pi}{2}-|\delta_H|.
		\end{align*}
		It follows by Corollary~\ref{cor:width} that $T_+-T_-\ge\omega_H(\Lambda)$. Second, assume that $T_--\delta_H\le A_-$: the same argument proves that
		\[T_+-T_-\ge\frac{\pi}{2}-|\delta_H|\ge\omega_H(\Lambda).\]
		
		The last case is when $A_-<T_-\le T_+<A_+$, namely $[T_-,T_+]\sq(A_-,A_+)$. Then, by Lemma~\ref{pro:tempodeg} and Lemma~\ref{lem:tempoconvex}, the set
		\[U:=\overline{I^-(\Sigma_{T_+})}\cap\overline{I^+(\Sigma_{T_-})}\]
		is a closed $H-$convex set containing $\Lambda$ and $\omega(U)=T_+-T_-$. By minimality, $U$ contains $\ch_H(\Lambda)$, hence 
		\[\arctan\left(\sup_\Sigma\lambda_1\right)-\arctan\left(\inf_\Sigma\lambda_n\right)=T_+-T_-=\omega(U)\ge\omega_H(\Lambda),\]
		concluding the proof.
	\end{proof}
	
	\begin{de}
		Let $\Sigma$ be a Riemannian manifold and let $B$ be a tensor of type $(1,1)$ on $\Sigma$. We denote the operator norm of $B(x)\in\mathrm{End}(T_x\Sigma,\pr{\cdot,\cdot})$ by
		\[\|B(x)\|^2:=\sup_{v\in T_x\Sigma}\frac{\pr{B(x)v,B(x)v}}{\pr{v,v}}.\]
	\end{de}
	We recall that $\|B(x)\|^2$ coincides with the maximal eigenvalue of $B(x){}^tB(x)$. In particular, $\|B(\cdot)\|\colon\Sigma\to\R$ is a continuous function.
	
	\begin{repcorx}{cor:w<B}
		Let $\Lambda$ be an admissible boundary in $\pd_\infty\hypu^{n,1}$. Let $B_0$ be the traceless shape operator of the properly embedded spacelike $H-$hypersurface $\Sigma$ such that $\pd_\infty\Sigma=\Lambda$. If \[\|B_0\|_{C^0(\Sigma)}^2\le 1+(H/n)^2,\] the width of $\ch_H(\Lambda)$ satisfies
		\[\tan\left(\omega_H(\Lambda)\right)\le\frac{2\|B_0\|_{C^0(\Sigma)}}{1+(H/n)^2-\|B_0\|_{C^0(\Sigma)}^2}.\]
	\end{repcorx}
	\begin{proof}
		By definition of traceless shape operator, we have \[\|B_0\|_{C^0(\Sigma)}=\max\left\{\|a_1\|_{C^0(\Sigma)},\|a_n\|_{C^0(\Sigma)}\right\},\] 
		for $a_i=\lambda_i-(H/n)$ the eigenvalues of $B_0$. In particular,
		\[\sup_\Sigma\lambda_1=\frac{H}{n}+\|a_1\|_{C^0(\Sigma)},\qquad\inf_\Sigma\lambda_n=\frac{H}{n}-\|a_n\|_{C^0(\Sigma)}\]
		By Theorem~\ref{pro:w<B}, it follows that 
		\begin{align*}
			\omega_H(\Lambda)&\le\arctan\left(\frac{H}{n}+\|a_1\|_{C^0(\Sigma)}\right)-\arctan\left(\frac{H}{n}-\|a_n\|_{C^0(\Sigma)}\right)\\
			&\le\arctan\left(\|B_0\|_{C^0(\Sigma)}+\frac{H}{n}\right)+\arctan\left(\|B_0\|_{C^0(\Sigma)}-\frac{H}{n}\right)\le\frac{\pi}{2}
		\end{align*}
		under the hypothesis $\|B_0\|_{C^0(\Sigma)}^2\le 1+(H/n)^2$. Since the tangent is strictly increasing over $[0,\pi/2]$, we obtain the result.
	\end{proof}
	
	\section{The width is an upper bound for the extrinsic curvature}\label{sec:B<w}
	The goal of this section is to prove that the width of the $H-$shifted convex hull is an upper bound of the norm of the traceless shape operator, namely an estimate going in the opposite direction with respect to the one found in Theorem~\ref{pro:w<B}. To be precise, we will show the following:
	\begin{repthmx}{thm:Schauder}
		For any $L\ge0$, there exists a universal constant $C_L$ with the following property: let $K\in[0,L]$ and $\Sigma$ be a properly embedded $H-$hypersurface in $\hyp^{n,1}$ with $H\in[K,L]$, and let $B_0$ be its traceless shape operator. Then,
		\[\|B_0\|_{C^0(\Sigma)}\le C_L\sin\left(\omega_K(\pd_\infty\Sigma)\right),\]
		for $\omega_K(\pd_\infty\Sigma)$ the width of the $K-$shifted convex hull of $\pd_\infty\Sigma$.
	\end{repthmx}
	Before going into details, we briefly explain the structure of this section. Let us fix an $H-$hypersurface $\Sigma$, a point $x\in\Sigma$ and a totally umbilical support $H-$hypersurface $\mathcal{P}_{\delta_H}$ for $\ch_H(\pd_\infty\Sigma)$. Let $v_H$ be the sine of the distance between $\Sigma$ and $\mathcal{P}_{\delta_H}$. 
	
	The function $v_H$ is bounded from above by the width $H-$shifted convex hull of $\Sigma$ (Propositio~\ref{pro:v<w}), and the second derivatives of $v_H$ approximate the traceless shape operator of $\Sigma$ around $x$ (Corollary~\ref{cor:hess}).
	
	It turns out that $v_H$ solves the elliptic PDE \[\Delta_\Sigma v_H-nv_H=f_H,\] for $\Delta_\Sigma$ the Laplace-Beltrami operator on $\Sigma$ (Propostion~\ref{pro:hess}). Using Schauder-type estimates, we bound the $C^2-$norm of $v_H$ with its $C^0-$norm, achieving our goal.
	
	The technical part lies in proving that the procedure does not depend on the choice of the $H-$hypersurface $\Sigma$, the point $x$ and the umbilical hypersurface $\mathcal{P}_{\delta_H}$. Under necessary but not restrictive assumptions on $\mathcal{P}_{\delta_H}$ (Definition~\ref{de:CMCLP}), we give bounds for the gradient (Proposition~\ref{pro:gradbound}) and for the Hessian (Corollary~\ref{cor:hess}) of $v_H$ over an open ball around $x$ of fixed radius (Corollary~\ref{cor:metricschaud}), which not depend on the choice of $\Sigma$ and $x$.
	
	Finally, we prove that the Laplace-Beltrami operators are uniformly elliptic over the space of $H-$hypersurfaces (Lemma~\ref{lem:beltrami}), hence Schauder estimates do not depend on the choice of $\Sigma$, concluding the proof of Theorem~\ref{thm:Schauder}.
	
	\subsection{Graphs over totally geodesic hypersurfaces}
	Let us fix a totally geodesic hypersurface $\mathcal{P}$, and denote by $e$ its past dual point, defined in Definition~\ref{de:Pp}. Let us fix an entire spacelike hypersurface $\Sigma$ and consider the function
	\[\begin{tikzcd}[row sep=1ex]
		u\colon\Sigma\arrow[r] & \R\\
		x\arrow[r,maps to] & \pr{x,e}
	\end{tikzcd}.\]
	
	From a geometric point of view, if $p\in I(e)$, \textit{i.e.} $p$ is time-related to $e$, then $\pr{\psi(p),\psi(e)}$ is the sine of the signed distance between $p$ and $\mathcal{P}$ (compare with Proposition~\ref{pro:distprod}). This function reflects the geometry of $\Sigma$, as already exploited in \cite{univ,andreamax}.
	
	\begin{pro}[Proposition~1.8 in \cite{andreamax}]\label{pro:hess0}
		\[
		\hess u-u\Id=\sqrt{1-u^2+|\nabla u|^2}B.
		\]
	\end{pro}
	\begin{proof}
		We have already remarked that $u$ is the restriction to $\Sigma$ of the function $U(p)=\pr{p,e}$, for $e$ the dual of the totally geodesic spacelike hypersurface $P$. A direct computation shows
		\begin{equation}\label{eq:gradU}
			\begin{split}
				\nablah U&=e+\pr{p,e}p\\
				\nabla u&=e+\pr{p,e}p+\pr{\nablah U,N}N=e+\pr{p,e}p+\pr{e,N}N.
			\end{split}
		\end{equation}
		It follows that
		\[
		\hess u(w)=\nabla_w\nabla u=\pr{e,p}w+\pr{e,N}B(w)=u(p)w+\pr{e,N}B(w).\]
		
		By Equation~\eqref{eq:gradU}, we have
		\begin{align*}
			|\nabla u|^2&=-1+u^2+\pr{e,N}^2,
		\end{align*}
		concluding the proof.
	\end{proof}
	
	The behaviour of CMC hypersurfaces is strictly linked to the umbilical hypersurfaces with the same mean curvature, which happens to be the equidistant hypersurfaces from a totally geodesic hypersurfaces. Hence, we define a new function encoding the signed distance from an umbilical hypersurface.
	
	Let us fix a spacelike totally geodesic hypersurface $\mathcal{P}$, its past dual point $e=\mathcal{P}^\perp_-$ and denote by $\mathcal{P}_\delta$ the hypersurface at signed distance $-\delta$ from $\mathcal{P}$. We remark that the mean curvature of $\mathcal{P}_\delta$ is $n\tan(\delta)$. Let us define the function
	\begin{equation}\label{eq:v}
		\begin{tikzcd}[row sep=1ex]
			v\colon\Sigma\arrow[r] & \R\\
			x\arrow[r,maps to] & \cos(\delta)u(x)+\sin(\delta)\sqrt{1-u(x)^2}
		\end{tikzcd},
	\end{equation}
	
	\begin{rem}\label{rem:geommeaning}
		Restricted to $\Sigma\cap I^+(e)$, the function $v$ equals $\sin\left(\dist(\cdot,\mathcal{P}_\delta)\right)$: indeed, for any $x\in\Sigma\cap I^+(e)$, we have 
		\[\dist(x,\mathcal{P})=\dist(x,\mathcal{P}_\delta)-\delta,\]
		since $\mathcal{P}_\delta$ is the set of points at constant distance $-\delta$ to $\mathcal{P}$. Then, we conclude 
		\begin{align*}
			\sin\left(\dist(x,\mathcal{P}_\delta)\right)&=\sin\left(\dist(x,\mathcal{P})+\delta\right)\\
			&=\cos(\delta)\sin\left(\dist(x,\mathcal{P})\right)+\sin(\delta)\sqrt{1-\sin\left(\dist(x,\mathcal{P})\right)^2}\\
			&=\cos(\delta)u(x)+\sin(\delta)\sqrt{1-u(x)^2}=v(x).
		\end{align*}
	\end{rem}
	
	Consider a $H-$hypersurface $\Sigma$ and a point $x\in\Sigma$. We want to study the geometry of $\Sigma$ in around $x$ through the function $v_H$ as in Equation~\eqref{eq:v}. Since $\Sigma$ is contained in $\ch_H(\Lambda)$ (Corollary~\ref{cor:Hmaxprin}) and the width of the $H-$shifted convex hull is at most $(\pi/2)-|\delta_H|$ (Corollary~\ref{cor:width}), we can choose $\mathcal{P}_{\delta_H}$ such that the distance from $x$ is at most $(\pi/4)-(|\delta_H|/2)$. We observe that $\mathcal{P}_{\delta_H}$ could lie in the past or in the future of $\Sigma$.
	
	In order to lighten the notation, we denote
	\begin{equation}\label{eq:UL}
		U_H(\mathcal{P}):=\left\{y\in\Sigma,|\dist(y,\mathcal{P}_{\delta_H})|\le\frac{\pi}{4}-\frac{|\delta_H|}{2}\right\}
	\end{equation} 
	and we introduce the following space.
	\begin{de}\label{de:CMCLP}
		Let $\mathcal{P}$ be a totally geodesic spacelike hypersurface, and $L\ge0$. We denote by $\cmclp(L,\mathcal{P})$ the space of properly embedded CMC hypersurfaces $\Sigma$ such that
		\begin{enumerate}
			\item $\Sigma$ has mean curvature $H\in[-L,L]$;
			\item $\mathcal{P}_{\delta_H}$ is a support hypersurface for $\ch_H(\pd_\infty\Sigma)$, for $\delta_H:=\arctan(H/n)$;
			\item $\Sigma\cap U_H(\mathcal{P})\ne\emptyset$.
		\end{enumerate}
	\end{de}
	
	Since the computations in the following have all a local nature, they will be carried over the open neighborhood of $U_H(\mathcal{P})$ given by
	\begin{equation}\label{eq:ULe}
		U_H(\mathcal{P},\varepsilon):=\left\{y\in\Sigma,|\dist(y,\mathcal{P}_{\delta_H})|<\frac{\pi}{4}-\frac{|\delta_H|}{2}+\varepsilon\right\},
	\end{equation}
	for a suitable choice of $\varepsilon$.
	
	Hereafter, we will assume at least $\varepsilon<\pi/4$, so that $U_H(\mathcal{P},\varepsilon)$ is contained in $I^+(e)$: as explained in Remark~\ref{rem:geommeaning}, this ensures that $v_H$ restricted to $U_H(\mathcal{P},\varepsilon)$ equals to $\sin(\cdot,\mathcal{P}_{\delta_H})$.
	
	\subsection{Gradient estimate}
	
	\begin{pro}\label{pro:hess}
		Let $\mathcal{P}$ be a totally geodesic spacelike hypersurface. For any spacelike hypersurface $\Sigma$, the function $v$ as in Equation~\eqref{eq:v}, restricted to $U_L(\mathcal{P},\varepsilon)$ (see Equation~\eqref{eq:ULe}) satisfies
		\[
		\hess v=v\Id+\sqrt{1-v^2+|\nabla v|^2}B-\frac{\tan(\delta)}{\sqrt{1-v^2}+v\tan(\delta)}\left(\Id+\frac{dv\nabla v}{1-v^2}\right).
		\]
	\end{pro}
	\begin{proof}
		A direct computation shows that
		\begin{equation}\label{eq:hess}
			\begin{split}
				u&=\cos(\delta)v-\sin(\delta)\sqrt{1-v^2}\\
				\nabla u&=\left(\cos(\delta)+\frac{\sin(\delta)}{\sqrt{1-v^2}}v\right)\nabla v\\
				\hess u&=\left(\cos(\delta)+\frac{\sin(\delta)}{\sqrt{1-v^2}}v\right)\hess v+\frac{\sin(\delta)}{(1-v^2)^{3/2}}dv\nabla v.
			\end{split}		
		\end{equation}
		
		We need to compare the last equation with Proposition~\ref{pro:hess0}, which states
		\[\hess u=u\Id+\sqrt{1-u^2+|\nabla u|^2}B,\]
		for $B$ the shape operator of $\Sigma$. One can check that
		\begin{align*}
			u&=\left(\cos(\delta)+\frac{\sin(\delta)}{\sqrt{1-v^2}}v\right)v-\frac{\sin(\delta)}{\sqrt{1-v^2}}\\
			1-u^2&=\left(\cos(\delta)+\frac{\sin(\delta)}{\sqrt{1-v^2}}v\right)^2(1-v^2),
		\end{align*}
		hence 
		\[\hess u=\left(\cos(\delta)+\frac{\sin(\delta)}{\sqrt{1-v^2}}v\right)\left(v\Id+\sqrt{1-v^2+|\nabla v|^2}B\right)-\frac{\sin(\delta)}{\sqrt{1-v^2}}.\]
		
		We conclude by comparing the above formula with the expression of $\hess u$ made in Equation~\eqref{eq:hess}, after the observation that
		\[\left(\cos(\delta)+\frac{\sin(\delta)}{\sqrt{1-v^2}}v\right)^{-1}\frac{\sin(\delta)}{\sqrt{1-v^2}}=\frac{\tan(\delta)}{\sqrt{1-v^2}+\tan(\delta)v}.\]
	\end{proof}
	
	We prove here a technical inequality which will be useful througout this section.
	\begin{lem}\label{lem:boundtan}
		Let $\mathcal{P}$ be a totally geodesic spacelike hypersurface, and $L\ge0$. For any $\Sigma$ and $H-$hypersurface in $\cmclp(L,\mathcal{P})$. Then
		\[\left|\sqrt{1-v_H^2}+v_H\tan(|\delta_H|)\right|\ge\frac{\cos\left(\frac{\pi}{4}+\frac{|\delta_H|}{2}+\varepsilon\right)}{\cos(|\delta_H|)}.\]
		over $U_H(\mathcal{P},\varepsilon)$, for $\varepsilon<(\pi/4)-(|\delta_H|/2)$. In particular, for $\varepsilon<(\pi/8)-(|\delta_H|/4)$, we have the uniform bound
		\[\left|\sqrt{1-v_H^2}+v_H\tan(|\delta_H|)\right|\ge\frac{1}{4}.\]
	\end{lem}
	\begin{proof}
		To lighten the notation, we assume $H\ge0$. Recalling that $v_H=\sin\left(\dist(\cdot,\mathcal{P}_{\delta_H})\right)$, one can check that
		\[
		\sqrt{1-v_H^2}+v_H\tan(\delta_H)=\frac{\sqrt{1-v_H^2}\cos(\delta_H)+v_H\sin(\delta_H)}{\cos(\delta_H)}=\frac{\cos\left(\delta_H-\dist(\cdot,\mathcal{P}_{\delta_H})\right)}{\cos(\delta_H)}.
		\]
		By choice of $\mathcal{P}_{\delta_H}$, we have a lower bound on its distance from $\Sigma$:
		\begin{align*}
			|\delta_H-\dist(\cdot,\mathcal{P}_{\delta_H})|&\le\delta_H+\frac{\pi}{4}-\frac{\delta_H}{2}+\varepsilon=\frac{\pi}{4}+\frac{\delta_H}{2}+\varepsilon<\frac{\pi}{2},
		\end{align*}
		for $\varepsilon<(\pi/4)-(\delta_H/2)$.	The cosine decreases in absolute value over $[-\pi/2,\pi/2]$. Hence,
		\[\left|\sqrt{1-v_H^2}+v_H\tan(\delta_H)\right|\ge\frac{\cos\left(\frac{\pi}{4}+\frac{\delta_H}{2}+\varepsilon\right)}{\cos(\delta_H)}.\]
		To conclude, assume $\varepsilon<(\pi/8)-(\delta_H/4)$, so that
		\begin{align*}
			\frac{\cos\left(\frac{\pi}{4}+\frac{\delta_H}{2}+\varepsilon\right)}{\cos(\delta_H)}>\frac{\cos\left(\frac{\pi}{2}+\frac{1}{4}\left(\delta_H-\frac{\pi}{2}\right)\right)}{\cos(\delta_H)}=-\frac{\sin\left(\frac{1}{4}\left(\delta_H-\frac{\pi}{2}\right)\right)}{\cos(\delta_H)}=:f(\delta_H).
		\end{align*}
		Over $[0,\pi/2]$, the function $f$ is decreasing: indeed, the sine is increasing over $[-\pi/8,0]$ and the cosine is decreasing over $[0,\pi/2]$. Hence, the minimum is achieved at $\pi/2$, namely
		\[f(\delta_H)>\lim_{t\to\pi/2} f\left(t\right)=\frac{1}{4},\]
		concluding the proof.
	\end{proof}
	
	\begin{pro}\label{pro:gradbound}
		For any $L\ge0$, there exists a universal constant $G_L>0$ with the following property. Let $\mathcal{P}$ be a totally geodesic spacelike hypersurface, and let $\Sigma\in\cmclp(L,\mathcal{P})$ be a $H-$hypersurface. The function $v_H$ as in Equation~\eqref{eq:v} satisfies
		\[|\nabla v_H|\le G_L|v_H|,\]
		over $U_H(\varepsilon)$, for $\varepsilon<(\pi/8)-(\delta_L/4)$.
	\end{pro}
	\begin{proof}
		We assume $H\ge0$ to lighten the notation. Let us fix an entire $H-$hypersurface $\Sigma$ and a point $p\in\Sigma$. Without loss of generality, one can assume $\nabla u(p)\ne0$ and consider the integral curve $\gamma$ of the (opposite) normalized gradient flow of $v$, \textit{i.e.} the solution of
		\[
		\begin{cases}
			\gamma(0)=p\\
			\gamma'(t)=-\frac{\nabla v_H}{|\nabla v_H|}
		\end{cases}.\] 
		Denoting $y(t):=|\nabla v_H\left(\gamma(t)\right)|$, the structure of the proof is the following: in the first step, we proof that if there exists a constant $A_H>0$ such that 
		\begin{equation}\label{eq:claimAB}
			\frac{d}{dt}y(t)\le A_H\sqrt{1+y(t)^2},
		\end{equation}
		then $\|\nabla v_H\|\le(A_H^2+2A_H)|v_H|$. In the second step, we actually produce the explicit (not necessarely sharp) constant
		\[A_H:=C(|H|,n)+n\left(1+4\tan(\delta_H)\right),\]
		for $C(|H|,n)$ as in \cite[Theorem~1]{kkn}, satisfing the inequality in Equation~\eqref{eq:claimAB}. 
		
		Hence, the universal constant $G_H:=A_H^2+2A_H$ satisfies the statement for $H-$hypersurfaces. Since the function $H\to G_H$ is increasing for $H\ge0$, the constant $G_L$ automatically works for any CMC hypersurface having mean curvature in $[0,L]$, concluding the proof.
		
		\step{Main argument} It follows from the claim that
		\[
		-A_H t\le\int_0^t\frac{y'(s)}{\sqrt{1+y(s)^2}}ds=\arcsinh\left(y(t)\right)-\arcsinh\left(y(0)\right).
		\]
		Applying the hyperbolic sine, which is increasing on $[0,+\infty)$, we obtain
		\begin{equation}\label{eq:sinh}
			-y(t)\le\sinh(A_H t)\sqrt{1+y(0)^2}-\cosh(A_H t)y(0).
		\end{equation}
		By construction of $\gamma$, we have
		\begin{align*}
			v_H\left(\gamma(t)\right)-v_H(p)&=\int_0^t\pr{\nabla v_H\left(\gamma'(s)\right),\gamma'(s)}ds\\
			&=\int_0^t\left\langle\nabla v_H\left(\gamma'(s)\right),-\frac{\nabla v_H\left(\gamma'(s)\right)}{|\nabla v_H\left(\gamma'(s)\right)|}\right\rangle ds=-\int_0^t y(s)ds.
		\end{align*}
		
		Integrating Equation~\eqref{eq:sinh}, one obtains
		\[v_H\left(\gamma(t)\right)-v_H(p)\le\frac{1}{A_H}\left((\cosh(A_H t)-1)\sqrt{1+y(0)^2}-\sinh(A_H t)y(0)\right)=:F(t).\]
		
		The function $F(t)$ admits a unique minimum over $[0,+\infty)$: by solving $F'(t)=0$, one finds that critical point solves
		\[\tanh(A_H t_{\min})=\frac{y(0)}{\sqrt{1+y(0)^2}},\] 
		hence
		\[F(t_{\min})=\frac{1}{A_H}\left(1-\sqrt{1+y(0)^2}\right)\]
		Now, we recall that $\mathcal{P}_{\delta_H}$ is a support $H-$umbilical hypersurface for $\Sigma$. By the strong maximum principle (Proposition~\ref{pro:maxprin}), either $\Sigma=\mathcal{P}_{\delta_H}$ or $\Sigma$ does not intersect $\mathcal{P}_{\delta_H}$. In the former case, the function $v_H$ identically vanishes, hence $\nabla v_H\equiv0$, concluding the proof. Otherwise, the function $v_H$ never vanishes. Up to a time-inverting isometry, we can assume $v_H>0$, hence $F(t_{\min})>-v_H(p)$. It follows, since $v_H$ takes values in $[-1,1]$, that
		\begin{align*}
			\frac{1}{A_H}\left(1-\sqrt{1+y(0)^2}\right)\ge-v_H(p) &\iff 1+y(0)^2\le A_H^2 v_H(p)^2+2A_H v_H(p)+1\\
			&\iff y(0)^2\le A_H^2\left(v_H(p)^2+2v_H(p)\right)\le(A_H^2+2A_H) v_H(p).
		\end{align*}
		Recalling that $y(0)^2=|\nabla v_H(p)|^2$ and $p$ was arbitrary, we found the seeken universal constant.
		
		\step{Producing $A_H$}	
		Observe that
		\begin{align*}
			y(t)\left|y'(t)\right|&=\frac{1}{2}\left|y'(t)^2\right|=|\pr{\nabla_{\gamma'(t)}\nabla v_H\left(\gamma(t)\right),\nabla v_H\left(\gamma(t)\right)}|\\
			&=|\pr{\hess v_H\left(\gamma'(t)\right),\nabla v_H\left(\gamma(t)\right)}|\le\|\hess v_H\|\,\left|\nabla v_H\left(\gamma(t)\right)\right|=\|\hess v_H\|\,y(t).
		\end{align*}
		Hence, to prove the claim it suffices to bound the norm of the Hessian of $v$. We recall that, by Proposition~\ref{pro:hess}, we have
		\begin{align*}
			\hess v_H=&v_H\Id+\sqrt{1-v_H^2+|\nabla v_H|^2}B-\frac{\tan(\delta_H)}{\sqrt{1-v_H^2}+v_H\tan(\delta_H)}\left(\Id+\frac{dv_H\nabla v_H}{1-v_H^2}\right).
		\end{align*}
		By \cite[Theorem~1]{kkn}, we have that $\|B\|$ is bounded from above by an explicit universal constant $C(|H|,n)$, Lemma~\ref{lem:boundtan} ensures that the denominator is bounded from below by $1/4$ and $|v_H|\le1$ by definition. Hence, an elementary computation shows that
		\[\left|\frac{d}{dt}y(t)\right|\le\|\hess v_H\|\le\underbrace{\left(n+C(|H|,n)+4n\tan(\delta_H)\right)}_{=:A_H}\sqrt{1+|\nabla v_H|^2}.\]
		satisfying the claim stated in Equation~\eqref{eq:claimAB}.	
	\end{proof}
	
	\begin{cor}\label{cor:tilt}
		For any $L\ge0$, there exists an angle $\alpha_L$ as follows: let $\mathcal{P}$ be totally geodesic spacelike hypersurface, and $\Sigma\in\cmclp(L,\mathcal{P})$. Let $\gamma$ be a timelike geodesic orthogonal to $\mathcal{P}_{\delta_H}$ and let $p:=\gamma\cap\Sigma$. Then, the angle between $\gamma$ and the timelike geodesic $\exp_p\left(tN(p)\right)$ is bounded by $\alpha_L$.
	\end{cor}
	\begin{proof}
		As we have repeatedly done so far, we fix $H\in[0,L]$, claim that
		\[\cosh(\alpha_H):=\left(1+\frac{1}{\cos\left(\frac{\pi}{4}+\frac{\delta_H}{2}\right)}\right)\sqrt{1+G_H^2},\] 
		and conclude by the fact that the map $H\to\alpha_H$ is increasing for $H\in[0,L]$.
		
		To have an explicit formula for the exponential map, we work in the quadric model $\hyp^{n,1}$. We denote by $P:=\psi(\mathcal{P})$ and abusively call $\Sigma$ the image of $\Sigma$ via $\psi$. Remark the $P_\delta$'s are the level sets of the function $U(\cdot)=\pr{e,\cdot}$, for $e$ the past dual point of the totally geodesic hypersurface $P$, namely $P=P_+(e)$. A timelike geodesic orthogonal to $P_{\delta}$ is then of the form $\gamma(t)=\cos(t)e+\sin(t)w$, for $w\in P_+(e)=P$ (compare with Equation~\eqref{eq:geod}).
		
		In particular, $\gamma$ is an integral line for $\nablah U$, hence the angle between $\gamma$ and $\Sigma$ coincides with $\cosh\pr{\nablah U,N}$. By Equation~\eqref{eq:gradU}, we obtain
		\begin{align*}
			\pr{\nablah U,N}=\pr{e,N}=\sqrt{1-u^2+|\nabla u|^2}=\left(\cos(\delta_H)+\frac{\sin(\delta_H)}{\sqrt{1-v_H^2}}\right)\sqrt{1-v_H^2+|\nabla v_H|^2}.
		\end{align*}
		
		By hypothesis, we have \[\sqrt{1-v_H^2}=\cos\left(\dist(\cdot,\mathcal{P}_{\delta_H})\right)\ge\cos\left(\frac{\pi}{4}+\frac{\delta_H}{2}\right).\]
		By Proposition~\ref{pro:gradbound}, we have $|\nabla v_H|^2\le G_H^2$, proving the claim and concluding the proof.
	\end{proof}
	\begin{cor}\label{cor:comp}
		The space $\cmclp(L,\mathcal{P})$ is compact, for any $L\in\R$, and for any totally geodesic spacelike hypersurface $\mathcal{P}$. 
	\end{cor}
	\begin{proof}
		Consider a sequence $(\Sigma_k)_{k\in\N}$ in $\cmclp(L,\mathcal{P})$. By Proposition~\ref{pro:compact}, we can extract a subsequence $(\Sigma_{k_j})_{j\in\N}$ which converges to an entire acausal hypersurface $\Sigma_\infty$, which is either a $H-$hypersurface or a degenerate hypersurface.
		
		Corollary~\ref{cor:tilt} prevents the normal vector of $\Sigma_k$ to degenerate, namely the latter case cannot happen. Since $\cmclp(L,\mathcal{P})$ is closed (compare with Definition~\ref{de:CMCLP}), $\Sigma_\infty$ belongs to $\cmclp(L,\mathcal{P})$, concluding the proof. 
	\end{proof}
	
	The last result of this section shows that $U_L(\mathcal{P},\varepsilon)$ contains all points closed enough to $U_L(\mathcal{P})$, with respect to the induced metric on $\Sigma$. 
	
	\begin{cor}\label{cor:metricschaud}
		For any $L\ge0$ and $\varepsilon\in\left(0,(\pi/8)-(\delta_L/4)\right)$, there exists a universal constant $R_L(\varepsilon)>0$ with the following property. Let $\mathcal{P}$ a totally geodesic spacelike hypersurface, for any $\Sigma\in\cmclp(L,\mathcal{P})$, we have
		\[B_\Sigma\left(U_L(\mathcal{P}),R_L(\varepsilon)\right)\sq U_L(\mathcal{P},\varepsilon),\]
		for $U_L(\mathcal{P}), U_L(\mathcal{P},\varepsilon)$ as in Equation~\eqref{eq:UL} and Equation~\eqref{eq:ULe}.
	\end{cor}
	\begin{proof}
		Let us fix $L\ge0$ $\varepsilon>0$ and $\Sigma\in\cmclp(L,\mathcal{P})$. For any pair $x,y\in\Sigma$ at distance $r$, let $\gamma$ be a length-minimizing geodesic connecting the two, whose existence is guaranteed by \cite[Theorem~B]{ecrin}. Assume $\gamma$ parameterized by arclength, then
		\begin{align*}
			|v_H(y)-v_H(x)|&=\left|\int_0^r\pr{\nabla v_H\left(\gamma(t)\right),\gamma'(t)}dt\nabla\right|\le\int_0^r \|\nabla v_H\|\|\gamma'\|\le G_L r
		\end{align*}
		for $G_L$ as in Proposition~\ref{pro:gradbound}.
		
		If $x\in U_L(\mathcal{P})$, namely $|\dist(x,\mathcal{P}_{\delta_H})|\le(\pi/2)-(|\delta_H|)/2$, then
		\begin{align*}
			\left|\dist(y,\mathcal{P}_{\delta_H})\right|&=\left|\arcsin\left(v_H(y)\right)\right|\le\left|\arcsin\left(v_H(x)+G_L r\right)\right|\\
			&\le\left|\dist(x,\mathcal{P}_{\delta_H})\right|+\sqrt{2}G_L r\le\frac{\pi}{4}-\frac{|\delta_H|}{2}+\sqrt{2}G_L r.
		\end{align*}
		It follows that $y\in U_L(\mathcal{P},\varepsilon)$ for \[r<\frac{\varepsilon}{\sqrt{2}G_L}=:R_L(\varepsilon),\] concluding the proof.
	\end{proof}
	
	We can finally (locally) bound the $C^0-$norm of $v$ with the width of the $H-$shifted convex hull.
	\begin{pro}\label{pro:v<w}
		For any $L\ge0$ and $\varepsilon\in\left(0,(\pi/8)-(\delta_L/4)\right)$, there exists a universal constant $E_L$ with the following property: for any spacelike $H-$hypersurface $\Sigma$ with $H\in[-L,L]$ and for any $x\in\Sigma$, there exists a support totally umbilical $H-$hypersurface for $\ch_H(\pd_\infty\Sigma)$ such that
		\[v_H(y)\le E_L\sin\left(\omega_H(\pd_\infty\Sigma)\right)\]
		over $B\left(x,R_L(\varepsilon)\right)$.
	\end{pro}
	\begin{proof}
		Let us fix $L\ge0$ and $\varepsilon\in\left(0,(\pi/8)-(\delta_L/4)\right)$, a spacelike $H-$hypersurface $\Sigma$ with $H\in[-L,L]$ and a point $x\in\Sigma$. 
		
		If $\Sigma$ is totally umbilical, there is nothing to prove: indeed, $\Sigma$ is a totally umbilical support $H-$hypersurface for $\ch_H(\pd_\infty\Sigma)$, and the corresponding function $v_H=\sin\left(\dist(\cdot,\Sigma)\right)$ identically vanishes, \textit{i.e.} any choice of $E_L$ satisfies the required inequality.
		
		Otherwise, there exists a totally umbilical support $H-$hypersurface $\mathcal{P}_{\delta_H}$ for $\ch_H(\pd_\infty\Sigma)$ such that \[v_H(x)\le\frac{1}{2}\sin\left(\omega_H(\pd_\infty\Sigma)\right).\]
		In particular, $|v_H(x)|\le(\pi/2)-(\delta_H/2)$, that is $x\in U_H(\mathcal{P})$.
		
		Let $y\in B_\Sigma\left(x,R_L(\varepsilon)\right)$, for $R_L(\varepsilon)$ as in Corollary~\ref{cor:metricschaud}, and let $\gamma$ be a geodesic of $\Sigma$ joining $x$ to $y$, parameterized by arclength so that $\gamma(0)=x$, $\gamma(r)=y$. Define $f(t):=v_H\left(\gamma(t)\right)$: by Proposition~\ref{pro:gradbound}, over $B_\Sigma\left(x,R_L(\varepsilon)\right)$ we have
		\[|f'(t)|=|\pr{\nabla v_H\left(\gamma(t)\right),\gamma'(t)}|\le G_L|v_H\left(\gamma(t)\right)|=G_L|f(t)|.\]
		Since $\mathcal{P}_{\delta_H}$ is a totally umbilical support $H-$hypersurface for $\ch_H(\pd_\infty\Sigma)$, by the strong maximum principle it cannot meet $\Sigma$, that is $f$ never vanishes over $\gamma$.	It follows that
		\begin{align*}
			\left|\ln\left(\frac{v(y)}{v(x)}\right)\right|&=\int_0^r\left|\frac{f'(t)}{f(t)}\right|dt\le G_L r\le G_L R_L(\varepsilon).
		\end{align*}
		Equivalently,
		\[|v_H(y)|\le e^{G_L R_L(\varepsilon)}|v_H(x)|\le e^{G_L R_L(\varepsilon)}\sin\left(\omega_H(\pd_\infty\Sigma)\right),\]
		concluding the proof.
	\end{proof}
	
	\subsection{Hessian estimate} Let $\Sigma$ be a properly embedded acausal hypersurface: a splitting naturally induces an embedding of $\Sigma$ in $\hypu^{n,1}$ as follows. Let $f\colon\hyp^n\to\R$ be the function such that $\Sigma=\gr f$. Then, define
	\[\begin{tikzcd}[row sep=1ex]
		\sigma\colon\hyp^n\arrow[r] &\hypu^{n,1}\\
		x\arrow[r,maps to] & \left(x,f(x)\right).
	\end{tikzcd}\]
	The precomposition by $\sigma$ induces a bijective correspondence between $C^{2}(\Sigma)$ and $C^{2}(\hyp^n)$.
	\begin{lem}\label{lem:hess}
		Let $L\ge0$, $R>0$. There exists a universal constant $D_L(R)>0$ with the following property: let $\mathcal{P}$ be a totally geodesic spacelike hypersurface and $\Sigma\in\cmclp(L,\mathcal{P})$. Consider a splitting $(p_0,\mathcal{P})$. Let $\sigma$ be the embedding of $\Sigma$ in $\hypu^{n,1}$ induced by the splitting. For any function $h\in C^2(\Sigma)$, we have
		\[\|\hess h\|_{\Sigma\cap\left(B_{\mathcal{P}}(p_0,R)\times\R\right)}\le D_L(R)\|h\circ\sigma\|_{C^2(B_{\mathcal{P}}(p_0,R))}.\]
	\end{lem}
	\begin{proof}
		In the coordinates induced by the splitting, the hessian of a function $h\in C^2(\Sigma)$ writes as
		\[\hess h=\sum_{i,j,k,l=1}^n\left(\frac{\pd^2(h\circ\sigma)}{\pd x_k\pd x_l}g^{ki}g^{lj}+\sum_{m=1}^n\frac{\pd(h\circ\sigma)}{\pd x_k}g^{kl}g^{jm}\Gamma_{lm}^i\right)\pd_i\otimes\pd_j.\]
		Christoffel's symbols depend on the coefficients $g_{ij}$ of the metric expressed, which can be expressed as smooth functions of $\sigma$. By compactness (Corollary~\ref{cor:comp}), we can bound all these quantities by a suitable constant $c_L(R)$ that does not depend on the choice of the hypersurface $\Sigma$. It follows that the norm of the hessian is bounded by a polynomial of degree 1 in the derivatives of $h$ up to the second order, which concludes the proof.
	\end{proof}
	
	Hereafter, in order to lighten the notation, we will denote by
	\[\|h\|_{C^2(B_{\mathcal{P}}(p_0,R))}:=\|h\circ\sigma\|_{C^2(B_{\mathcal{P}}(p_0,R))}.\]
	\begin{cor}\label{cor:hess}
		For any $L\ge0$ and $\varepsilon\in\left(0,(\pi/8)-(\delta_L/4)\right)$, there exists a universal constant $A_L(\varepsilon)>0$ with the following property: let $\mathcal{P}$ be a totally geodesic spacelike hypersurface and let $\Sigma\in\cmclp(L,\mathcal{P})$ be a $H-$hypersurface. For any $p\in U_H(\mathcal{P})$ there exists a point $p_0\in\mathcal{P}$ such that the function $v_H$ as in Equation~\eqref{eq:v} satisfies
		\[\|B_0\|_{C^0\left(\Sigma\cap(B_{\mathcal{P}}(p_0,R_L(\varepsilon))\times\R)\right)}\le A_L(\varepsilon)\|v_H\|_{C^2(B_{\mathcal{P}}(p_0,R_L(\varepsilon)))},\]
		for $R_L(\varepsilon)$ as in Corollary~\ref{cor:metricschaud} $B_0$ the traceless shape operator of $\Sigma$.
	\end{cor}
	\begin{proof}
		Let us fix $L\ge0$ and $\varepsilon\in\left(0,(\pi/8)-(\delta_L/4)\right)$. To lighten the notation, assume $H\ge0$ and denote by $\delta_H=\arctan(H/n)$.
		We claim that the (non-sharp) constant
		\[
		A_H(\varepsilon):=\sqrt{2}\left(D_H\left(R_L(\varepsilon)\right)+n+2^{3/4}\tan(\delta_H)G_H^2+4n\left(1+G_H^2+G_H\tan(\delta_H)\right)\right)
		\]
		satisfies the statement for $H-$hypersurfaces. Since the functions $H\mapsto G_H$ and $H\mapsto\delta_H$ are increasing for $H\ge0$, the function $H\mapsto A_H$ is increasing, too. It follows that if the claim is satisfied for $H-$hypersurfaces, it is automatically satisfied for any CMC hypersurface having mean curvature in $[0,H]$, concluding the proof.
		
		Let us fix a $H-$hypersurface $\Sigma\in\cmclp(L,\mathcal{P})$. Since the shape operator of $\Sigma$ is $B=B_0+\tan(\delta_H)\Id$, Proposition~\ref{pro:hess} gives
		\begin{align*}
			\sqrt{1-v_H^2+|\nabla v_H|^2}B_0=&\hess v_H-v_H\Id+\frac{dv_H\nabla v_H}{(1-v_H^2)^{3/2}}\tan(\delta_H)\\
			&+\left(\frac{1}{\sqrt{1-v_H^2}+v_H\tan(\delta_H)}-\sqrt{1-v_H^2+|\nabla v_H|^2}\right)\tan(\delta_H)\Id.
		\end{align*}
		
		Since $p\in U_H(\mathcal{P})$, its distance from $\mathcal{P}_{\delta_H}$ is less than $(\pi/4)-(\delta_H/2)$, hence
		\[\sqrt{1-v_H^2+|\nabla v_H|^2}\ge\sqrt{1-v_H^2}\ge\cos\left(\frac{\pi}{4}-\frac{\delta_H}{2}\right)\ge\frac{1}{\sqrt{2}}.\]
		It follows that
		\[\|B_0\|\le\sqrt{2}\sqrt{1-v_H^2+|\nabla v_H|^2}\|B_0\|.\]
		Moreover, by Lemma~\ref{lem:hess} and Proposition~\ref{pro:gradbound}, we get
		\begin{equation}\label{eq:stimaf1}
			\begin{split}
				&\left\|\hess v_H-v_H\Id+\frac{dv_H\nabla v_H}{(1-v_H^2)^{3/2}}\tan(\delta_H)\right\|\\
				&\qquad\le D_L\left(R_L(\varepsilon)\right)\|v_H\|_{C^2}+n\|v_H\|_{C^0}+2^{3/4}\tan(\delta_H)G_H^2\|v_H\|^2_{C^0}\\
				&\qquad\le\left(D_L\left(R_L(\varepsilon)\right)+n+2^{3/4}\tan(\delta_H)G_H^2\right)\|v_H\|_{C^2}.
			\end{split}
		\end{equation}
		In the last line, we used that $\|v_H\|_{C^0}^2\le\|v_H\|_{C^0}$ since $|v_H|\le1$ over $U_H(\mathcal{P},\varepsilon)$, for $\varepsilon<\pi/4$.
		
		To conclude, use Lemma~\ref{lem:boundtan} and Proposition~\ref{pro:gradbound} to get
		\begin{equation}\label{eq:stimaf2}
			\begin{split}
				&\left\|\frac{1}{\sqrt{1-v_H^2}+v_H\tan(\delta_H)}-\sqrt{1-v_H^2+|\nabla v_H|^2}\right\|_{C^0}\\
				&\qquad\le4\left\|1-\sqrt{1-v_H^2+|\nabla v_H|^2}+\sqrt{1-v_H^2+|\nabla v_H|^2}v_H\tan(\delta_H)\right\|_{C^0}\\
				&\qquad\le4\left\|v_H^2+|\nabla v_H|^2+G_H v_H\tan(\delta_H)\right\|_{C^0}\\
				&\qquad\le4\left(1+G_H^2+G_H\tan(\delta_H)\right)\|v_H\|_{C^0},
			\end{split}
		\end{equation}
		proving the claim and concluding the proof.
	\end{proof}
	
	\begin{rem}
		The fact that the $C^2$ behaviour of $v_H$ encodes the curvature of $\Sigma$ should not surprise. The relevance of the statement is due to the universality of the constant $A_L(\varepsilon)$.
	\end{rem}
	
	\subsection{Schauder estimate}
	At this point, we showed that for a $H-$hypersurface, the norm of the traceless shape operator is a big O of the $C^2-$norm of the distance from a totally umbilical spacelike hypersurface with the same mean curvature. Through Schauder's methods, we prove that the $C^2-$norm of $v_H$ is in fact uniformly bounded by its $C^0-$norm.
	
	The $C^0-$norm of $v_H$ is related to the width of the $H-$shifted convex hull, since we always choose $v_H$ to be less than $\sin(\omega_H/2)$, while $B_0$ depends on the $C^2-$behaviour of $v_H$.
	
	The plan is to pullback the problem on $\hyp^n$ by projecting the graph of $\Sigma$ on a suitable splitting, in order to apply elliptic PDE's theory.
	
	We state a key result for this section, coming from elliptic PDE's theory, in an easier version. The original result can be found in \cite[Theorem~6.2]{pde}.
	\begin{pro}\label{pro:pde}
		Let $v\in C^{2}\left(B(0,R)\right)$ be a bounded solution of the elliptic PDE
		\[Lv=\sum_{i,j=1}^n a_{ij}\frac{\pd^2}{\pd x_i\pd x_j}v+\sum_{j=1}^nb_j\frac{\pd}{\pd x_j}v-nv=f,\]
		for $f\in C^0\left(B(0,R)\right)$. Assume the tensor $A=(a_{ij})$ is uniformly positive definite and that its coefficients are uniformly bounded, \textit{i.e.} there exist positive constants $\lambda,\Lambda>0$ such that 
		\begin{align*}
			A\ge\lambda\Id, &&\|a_{ij}\|_{C^{0}\left(B(0,R)\right)},\|b_{j}\|_{C^{0}\left(B(0,R)\right)}\le\Lambda.
		\end{align*}
		
		Then, there exists a constant $C=C(n,\lambda,\Lambda,R)$ such that
		\[\|v\|_{C^2\left(B(0,R/2)\right)}\le C\left(\|v\|_{C^0\left(B(0,R)\right)}+\|f\|_{C^0\left(B(0,R)\right)}\right).\]	
	\end{pro}  
	By tracing the expression of the Hessian in Proposition~\ref{pro:hess}, we obtain that $v_H$ solves the elliptic PDE $\Delta v_H-nv_H=f_H$, for $\Delta$ the Laplace-Beltrami operator on $\Sigma$ and 
	\[f_H=-\tan(\delta_H)\left(\frac{|\nabla v_H|^2}{(1-v_H^2)^{3/2}}-n\sqrt{1-v_H^2+|\nabla v_H|^2}+\frac{n}{\sqrt{1-v_H^2}+v_H\tan(\delta_H)}\right).\]
	
	First, we prove that the Laplace-Beltrami operator on $\Sigma$ satisfies the conditions of Proposition~\ref{pro:pde} and that the constant $C$ does not depend on $\Sigma$, for a suitable $R>0$.
	
	\begin{lem}\label{lem:beltrami}
		Let $L\ge0$ and $R>0$. There exist constants $\lambda_L(R),\Lambda_L(R)$ as follows. Let $(q_0,\mathcal{P})$ be a splitting, and let $\Sigma\in\cmclp(L,\mathcal{P})$ be a $H-$hypersurface and let
		\[\Delta_\Sigma:=\sum_{i,j=1}^n a_{ij}\frac{\pd^2}{\pd x_i\pd x_j}+\sum_{j=1}^nb_j\frac{\pd}{\pd x_j}\]
		be the Laplace-Beltrami operator of $\Sigma$ in the coordinates given by the splitting $(q_0,\mathcal{P})$.
		
		Let $p_0$ be the (unique) point of $\Sigma$ contained in the fiber $\{x_0\}\times\R$. If $p_0$ is contained in $U_H(\mathcal{P})$ then, $A\ge\lambda_L(R)\Id$ and
		\[\|a_{ij}\|_{C^{0}\left(B(x_0,R)\right)},\|b_{j}\|_{C^{0}\left(B(x_0,R)\right)}\le\Lambda_L(R).\]
	\end{lem}	
	In other words, for any pointed $H-$hypersurface $(p_0,\Sigma)$, we can find a splitting $(q_0,\mathcal{P})$ such that the Laplace-Beltrami operator of $\Sigma$ is a uniformly elliptic operator around $p_0$, and the constants of ellipticity do not depend on the hypersurface nor on the point.
	\begin{proof}
		By Corollary~\ref{cor:comp}, the space $\cmclp(L,\mathcal{P})$ is compact with respect to the Hausdorff topology, which is equivalent to the $C_0^\infty(\hyp^n)$ on spacelike CMC graph by Proposition~\ref{pro:compact}. It follows that $\cmclp(L,\mathcal{P})$ embedds as a compact subset of $C^\infty\left(\overline{B_{\hyp^n}(x_0,R)}\right)$. 
		
		Pick $\Sigma\in\cmclp(L,\mathcal{P})$ and let $\Sigma=\gr f$ in the splitting $(q_0,\mathcal{P})$. Then,
		\[p_0=\left(x_0,f(x_0)\right)\in U_H(\mathcal{P})\iff |f(x_0)-\delta_H|\le\frac{\pi}{4}-\frac{\delta_H}{2}.\]
		
		Hence, the CMC satisfying this technical property are a compact subset of $\cmclp(L,\mathcal{P})$ inside $C^\infty_0(\hyp^n)$. The Laplace-Beltrami operator of $\Sigma$ can be explicitly written as a smooth function of $f$, its first and second order derivatives. Hence, the coefficients of $\Delta_\Sigma$ are uniformly bounded over $\overline{B_{\hyp^n}\left(x_0,R\right)}$, and the bound does not depend on $\Sigma$, but only on $L$ and $R$, which concludes the proof.
	\end{proof}
	
	We are finally ready to prove the main result of this subsection.
	\begin{repthmx}{thm:Schauder}
		Let $L\ge K\ge0$. There exists a universal constant $C_L$ with the following property: let $\Sigma$ a properly embedded $H-$hypersurface in $\hyp^{n,1}$ with $H\in[K,L]$, and let $B_0$ be its traceless shape operator. Then,
		\[\|B_0\|_{C^0(\Sigma)}\le C_L\sin\left(\omega_K(\pd_\infty\Sigma)\right).\]
	\end{repthmx}
	\begin{proof}
		Let us fix a properly embedded CMC hypersurface $\Sigma$, with $H\in[K,L]$, and $p_0\in\Sigma$. Let us also fix \[\varepsilon<\frac{\pi}{8}-\frac{\delta_L}{4}\le\frac{\pi}{8}-\frac{\delta_H}{4},\]
		and denote by  $R=R_L(\varepsilon)$ as in Corollary~\ref{cor:metricschaud}.
		
		Take a support $H-$umbilical hypersurface $\mathcal{P}_{\delta_H}$ as in Proposition~\ref{pro:v<w}. In particular,
		\[|\dist(p_0,\mathcal{P}_{\delta_H})|\le\frac{\omega_H(\Lambda)}{2}\le\frac{\pi}{4}-\frac{|\delta_H|}{4},\]
		the last inequality coming from Corollary~\ref{cor:width}. 
		
		In other words, $p_0\in U_H(\mathcal{P})$, that is $\Sigma\in\cmclp(L,\mathcal{P})$, for $\mathcal{P}$ the totally geodesic spacelike hypersurface equidistant to $\mathcal{P}_{\delta_H}$. Let $q_0\in \mathcal{P}$ a point realizing the distance $\dist(p_0,\mathcal{P})$. Then, in the splitting $(q_0,\mathcal{P})$, the point belongs to the fiber $\{x_0\}\times\R$.
		
		By tracing the expression of the Hessian find in Proposition~\ref{pro:hess}, we obtain that $v_H$ solves the elliptic PDE $\Delta_\Sigma v_H-nv_H=f_H$, for $\Delta_\Sigma$ the Laplace-Beltrami operator on $\Sigma$ and 
		\[f_H=-\tan(\delta_H)\left(\frac{|\nabla v_H|^2}{(1-v_H^2)^{3/2}}-n\sqrt{1-v_H^2+|\nabla v_H|^2}+\frac{n}{\sqrt{1-v_H^2}+v_H\tan(\delta_H)}\right).\]
		
		By Lemma~\ref{lem:beltrami}, the Beltrami-Laplacian operator of $\Sigma$ is an elliptic operator with coefficients bounded by $\lambda_L(R),\Lambda_L(R)$ over $B(x_0,R)$. Hence, by Schauder interior estimates (Proposition~\ref{pro:pde}), there exists a universal constant 
		\[c_L:=C(n,\lambda_L(R),\Lambda_L(R),R)\] 
		such that, for any solution $L_\Sigma v=f$, we have
		\[\|v\|_{C^2\left(B(0,R/2)\right)}\le c_L\left(\|v\|_{C^0\left(B(0,R)\right)}+\|f\|_{C^0\left(B(0,R)\right)}\right).\]
		
		Comparing with the proof of Corollary~\ref{cor:hess}, and more precisely with Equation~\eqref{eq:stimaf1} and Equation~\eqref{eq:stimaf2}, we get \[\|f_H\|_{C^0\left(B(0,R)\right)}\le nA_L(R)\|v_H\|_{C^0\left(B(0,R)\right)}.\]
		Hence, denoting by $C'_L:=c_L A_L(R)\left(1+nA_L(R)\right)$, we have
		\[\|B_0(p_0)\|\le\|B_0\|_{C^2\left(B(0,R/2)\right)}\le C'_L\|v_H\|_{C^0\left(B(0,R)\right)}\le C'_L E_L(\varepsilon)\sin\left(\omega_H(\pd_\infty\Sigma)\right),\]
		for $E_L(\varepsilon)$ as in Proposition~\ref{pro:v<w}.
		
		Finally, let us denote by $C_L:=C'_L E_L(\varepsilon)$: by Lemma~\ref{lem:width}, we have
		\[\|B_0(p_0)\|\le C_L\omega_H(\pd_\infty\Sigma)\le C_L\omega_K(\pd_\infty\Sigma),\] which concludes the proof since the choice of $p_0$ was arbitrary.
	\end{proof}
	
	\section{Application I: sectional curvature}\label{sec:sectional}
	This brief section wants to stress the link between the the width of the $H-$shifted convex hull of an admissible boundary $\Lambda$ and the sectional curvature of the corresponding $H-$hypersurface.
	
	Let $\Sigma$ be a spacelike hypersurface in the Anti-de Sitter space. Gauss equation allows to compute the sectional curvature of $\Sigma$ through its shape operator. Indeed, let $v,w\in T_x\Sigma$ two orthonormal vectors, then
	\begin{align*}
		K_\Sigma\left(\Span(v,w)\right)&=-1-\sff(v,v)\sff(w,w)+\sff(v,w)^2\\
		&=-1-\pr{B(v),v}\pr{B(w),w}+\pr{B(v),w}^2.
	\end{align*}
	The orthonormal basis $(v,w)$ can be chosen so that $\pr{B(v),w}=0$, as a consequence of the spectral theorem. It follows that 
	\begin{align*}
		-K_\Sigma\left(\Span(v,w)\right)-1&=\pr{B(v),v}\pr{B(w),w}=\left(\pr{B_0(v),v}+(H/n)\right)\left(\pr{B_0(w),w}+(H/n)\right)\\
		&=\pr{B_0(v),v}\pr{B_0(w),w}+(H/n)^2+(H/n)\pr{B_0(v+w),v+w}\\
		&\ge-\|B_0\|(x)^2+(H/n)^2-2(|H|/n)\|B_0\|(x).
	\end{align*}
	In the second line we used that $\pr{B(v),w}=0$ is equivalent to $\pr{B_0(v),w}=0$.
	
	Hence, the sectional curvature of $\Sigma$ is uniformly bounded by an explicit function of the norm of the traceless operator:
	\begin{equation}\label{eq:gauss}
		\begin{split}
			\max_{\mathrm{Gr}_2(T_{x}\Sigma)}K_\Sigma\le-1-(H/n)^2+\|B_0(x)\|^2+2(H/n)\|B_0(x)\|.
		\end{split}
	\end{equation}
	
	To our knowledge, it is still an open question wheter CMC hypersurfaces are Hadamard manifolds in higher dimension. However, a direct consequence of Theorem~\ref{thm:Schauder} recovers the existence of many CMC hypersurfaces in Anti-de Sitter space with \textit{uniform} negative sectional curvature.
	
	\begin{repcorx}{cor:sectional}
		For any $H\in\R$, there exists a universal constant $K_H>0$ such that
		\[\sup_{\mathrm{Gr}_2(T\Sigma)}K_\Sigma<-1-\left(\frac{H}{n}\right)^2+K_H\sin\left(\omega_H(\pd_\infty\Sigma)\right),\]
		for any properly embedded $H-$hypersurface $\Sigma$ in $\hyp^{n,1}$.
	\end{repcorx}
	\begin{proof}
		The proof consists in comparing Equation~\eqref{eq:gauss} with the statement of Theorem~\ref{thm:Schauder}. 
		
		Indeed, the supremum of the sectional curvature of $\Sigma$ is bounded by
		\begin{align*}\sup_{\mathrm{Gr}_2(T\Sigma)}K_\Sigma&\le-1-\left(\frac{H}{n}\right)^2+\frac{2|H|}{n}\|B_0\|_{C^0(\Sigma)}+\|B_0\|^2_{C^0(\Sigma)}\\
			&\le-1-\left(\frac{H}{n}\right)^2+\left(\frac{2|H|}{n}+\|B_0\|_{C^0(\Sigma)}\right)\|B_0\|_{C^0(\Sigma)}.
		\end{align*}
		By \cite[Theorem~1]{kkn}, the norm of $B$ is bounded by a uniform constant $C(|H|,n)$: hence, $(2|H|/n)+\|B_0\|_{C^0(\Sigma)}$ is bounded by a constant $K'_H$, too. By Theorem~\ref{thm:Schauder}, we have
		\begin{align*}\sup_{\mathrm{Gr}_2(T\Sigma)}K_\Sigma&-1-(H/n)^2+K'_H\|B_0\|_{C^0(\Sigma)}\\
			&\le-1-(H/n)^2+K'_H C_H\sin\left(\omega_H(\Lambda)\right).
		\end{align*}
		Then, the constant $K_H:=K'_H C_H$ satisfies the statement, concluding the proof.
	\end{proof} 
	Unfortunately, this result cannot be state in terms of neighbourhoods of totally geodesics boundaries with respect to the Hausdorff topology: in the $2-$dimensional case, it is known that if an admissible boundary $\Lambda$ contains two transverse lightlike rays, then the width of its convex hull is $\pi/2$. Conversely, in \cite{alex}, a class of admissible boundaries in $\hyp^{2,1}$ whose associated maximal surface is asymptotically flat is presented: such a class is dense in the space of admissible boundaries with respect to the Hausdorff topology.
	
	For $n=2$, the sectional curvature of CMC surfaces is non-positive. For the sake of completeness we add an \textit{ad hoc} proof, which need a preliminary result, peculiar of surfaces, which is proved for example in \cite[Lemma~3.11]{particles} or in \cite[Lemma~3.1]{chrisandrea}:
	\begin{lem}
		Let $(\Sigma,g)$ be a Riemannian $2-$manifold, and let $B_0$ be a traceless, $g-$symmetric and $g-$Codazzi $(1,1)-$smooth tensor. Denote by $\chi:=\log(-\det B_0)$, then
		\[\frac{1}{4}\Delta^g\chi=K_\Sigma,\] 
		for $\Delta^g \chi:=\tr(\mathrm{Hess}\,\chi)$ the Laplace-Beltrami operator on $\Sigma$.
	\end{lem}
	\begin{lem}\label{lem:curv}
		Let $\Sigma$ be a properly embedded spacelike CMC hypersurface in $\hyp^{2,1}$, then the sectional curvature of $\Sigma$ is non-positive.
	\end{lem}
	\begin{proof}
		The shape operator $B$ of $\Sigma$ is automatically a $g-$symmetric and $g-$Codazzi $(1,1)-$tensor. Since $\Sigma$ has constant mean curvature $H$, the same yelds for the traceless shape operator $B_0:=B-H\Id$. By Gauss equation, \begin{equation}\label{eq:gauss3d}
			K_\Sigma=-1-\det B=-1-(H/2)^2-\det B_0.
		\end{equation}
		It follows that
		\[\frac{1}{4}\Delta^g\chi=-1+(H/2)^2-\det B_0=e^\chi-1-(H/2)^2.\]
		By the strong maximum principle, we have $e^\chi\le 1+(H/2)^2$, that is $K_\Sigma$ is either strictly less than $0$ or identically vanishes.
	\end{proof}
	Hence, in the $3-$dimensional case, Corollary~\ref{cor:w<B} can be stated in terms of the sectional curvature of $\Sigma$. 
	\begin{repcorx}{lem:sect}
		Let $\Lambda$ be an admissible boundary in $\pd_\infty\hypu^{2,1}$. Let $B_0$ be the traceless shape operator of the properly embedded spacelike $H-$hypersurface such that $\pd_\infty\Sigma=\Lambda$. Then,
		\[\tan\left(\omega_H(\Lambda)\right)\le-\frac{2\|B_0\|_{C^0(\Sigma)}}{\sup_{\Sigma}K_\Sigma}.\]
	\end{repcorx}
	\begin{proof}
		In the $3-$dimensional case, the eigenvalues of the traceless shape operator are opposite, hence $\|B_0\|=\|a_1\|_{C^0(\Sigma)}=\|a_2\|_{C^0(\Sigma)}$. In particular, by Equation~\eqref{eq:gauss3d}, we have $\|B_0\|^2\le1+(H/2)^2$. Hence, by Corollary~\ref{cor:w<B}, we have
		\[\tan\omega_H(\Lambda)\le\frac{2\|B_0\|_{C^0(\Sigma)}}{1+(H/2)^2-\|B_0\|^2}=-\frac{2\|B_0\|_{C^0(\Sigma)}}{\sup_{\Sigma}K_\Sigma},\]
		concluding the proof.
	\end{proof}	
	
	\section{Application II: Teichm\"uller theory}\label{sec:teich}
In this chapter, we focus on \textit{quasiconformal} (Definition~\ref{de:quasiconf}) $\theta-$landslides. They have been studied in \cite{areapres} from a qualitative point of view. We are interested in a quantitative investigation: the \textit{quasiconformal dilatation} $\mathcal{K}(\Phi)$ (Equation~\eqref{eq:maxdil}) of a $\theta-$landslide $\Phi$ is bounded by the principal curvatures of the CMC surface whose asymptotic boundary is the graph of $\Phi|_{\pd_\infty\hyp^2}$, as already remarked in \cite{tamb}. As a consequence of consequence of Theorem~\ref{thm:Schauder}, we bound $\mathcal{K}(\Phi)$ with a multiple of the $H-$width, which is related to the \textit{cross-ratio norm} of $\Phi|_{\pd_\infty\hyp^2}$ thanks to the work \cite{andreamax}.
\subsection{Universal Teichm\"uller space}
The \textit{cross-ratio} of a quadruple $(z_1,z_2,z_3,z_4)$ of points in $\pd_\infty\hyp^2$ is
\[cr(z_1,z_2,z_3,z_4):=\frac{z_4-z_1}{z_2-z_1}\frac{z_3-z_2}{z_3-z_4}.\] 
\begin{de}\label{de:quasisym}
	Let $\phi\colon\pd_\infty\hyp^2\to\pd_\infty\hyp^2$ be an orientation preserving homeomorphism. The \textit{cross-ratio norm} of $\phi$ is
	\[\|\phi\|_{cr}:=\sup_{cr(z_1,z_2,z_3,z_4)=-1}\ln\left|cr\left(\phi(z_1),\phi(z_2),\phi(z_3),\phi(z_4)\right)\right|\in[0,+\infty].\]
	
	We call $\phi$ is \textit{quasi-symmetric} if $\|\phi\|_{cr}<+\infty$.
\end{de}
\begin{rem}
	By associating a geodesic to its end points, the space of oriented geodesics of $\hyp^2$ identifies with $\left(\pd_\infty\hyp^2\times\pd_\infty\hyp^2\right)\setminus\Delta$. One can check that \[cr(z_1,z_2,z_3,z_4)=-1\] if and only if the geodesics $(z_1,z_3)$ and $(z_2,z_4)$ are orthogonal. The geometric meaning of quasi-symmetric homeomorphism is then that the angle between the geodesics $\left(\phi(z_1),\phi(z_3)\right)$ and $\left(\phi(z_2),\phi(z_4)\right)$ is uniformly bounded, for any couple of orthogonal geodesics $(z_1,z_3),(z_2,z_4)$.
\end{rem}

The \textit{universal Teichm\"uller space} is the space of quasi-symmetric homeomorphisms of the circle, \textit{modulo} the action of $\psl$ by postcomposition.

The boundary of $\hyp^{2,1}$ is a quadric ruled by two family of lightlike lines. For each point $x\in\pd_\infty\hyp^{2,1}$, we call $l_x$ (resp. $r_x$) the lightlike line through $x$ belonging to the first (resp. the second) family. Let us introduce a new parameterization of $\pd_\infty\hyp^{2,1}$: we fix a totally geodesic spacelike plane $P\cong\hyp^2$, and define
\[
\begin{tikzcd}[row sep=1ex]
	\left(\pi_l,\pi_r\right)\colon\pd_\infty\hyp^{2,1}\arrow[r]&\pd_\infty P\times\pd_\infty P\cong\pd_\infty\hyp^2\times\pd_\infty\hyp^2\\
	\qquad\quad x\arrow[r,mapsto]&\left(l_x\cap\pd_\infty P,r_x\cap\pd_\infty P\right)
\end{tikzcd}.\]
It turns out that, via this identification, both orientation-preserving homemorphisms and quasi-symmetric homemorphisms of $\pd_\infty\hyp^2$ can be characterized purely in terms of Anti-de Sitter geometry, as proved in \cite[Theorem~1.12]{univ}:
\begin{pro}\label{pro:bonsch}
	A subset $\Lambda$ in $\pd_\infty\hyp^{2,1}$ is an acausal admissible boundary if and only if the map \[\phi:=\pi_r\circ\pi_l^{-1}\colon\pi_l(\Lambda)\sq\pd_\infty\hyp^2\to\pd_\infty\hyp^2\]
	is well defined and an orientation preserving homeomorphism.
	
	Moreover, $\phi$ is quasi-symmetric if and only if $\omega_0(\Lambda)<\pi/2$.
\end{pro}
In \cite[Proposition~3.A]{andreamax}, this result has been improved, from a quantitative point of view:
\begin{lem}\label{lem:cross}
	Let $\phi\colon\pd_\infty\hyp^2\to\pd_\infty\hyp^2$ be a orientation-preserving homeomorphism, and denote by $\Lambda$ its graph. Then,
	\[\tan\left(\omega_0(\Lambda)\right)\le\sinh\left(\frac{\|\phi\|_{cr}}{2}\right).\] 
\end{lem}

\subsection{Quasiconformal maps} Let $U\sq\C$ be a domain and $f\colon U\to\C$ an orientation preserving homeomorphism onto its image. 

\begin{de}\label{de:quasiconf}
	The map $f\colon U\to\C$ is \textit{quasiconformal} if $f$ is absolutely continuous over lines and there exists a constant $k<1$ such that
	\[|\pd_{\bar{z}}f|\le k|\pd_z f|\]
	almost everywhere.
\end{de}
The function $\mu_f:=\pd_{\bar{z}}f/\pd_z f$, which is defined almost-everywhere, is called \textit{complex dilatation} of $f$. From a geometric point of view, $f$ is quasiconformal if and only the distortion of \[\sph^1=T^1_z\hyp^2\sq T_z\hyp^2\] through $d_zf$ is uniformely bounded as $z$ varies over $\hyp^2$. Indeed, one can prove (see for example \cite{ahlquasi}) that the ratio between the major and the minor axis of the ellipse $d_z f(\sph^1)$ is precisely
\[\frac{\left|d_z f(\pd_z)\right|+\left|d_z f(\pd_{\bar{z}})\right|}{\left|d_z f(\pd_z)\right|+\left|d_z f(\pd_{\bar{z}})\right|}=\frac{1+|\mu_f|}{1-|\mu_f|}\le\frac{1+\|\mu_f\|_{L^\infty(\hyp^2)}}{1-\|\mu_f\|_{L^\infty(\hyp^2)}}\]
The number
\begin{equation}\label{eq:maxdil}
	\mathcal{K}(f):=\frac{1+\|\mu_f\|_{L^\infty(\hyp^2)}}{1-\|\mu_f\|_{L^\infty(\hyp^2)}}
\end{equation}
is the \textit{maximal dilatation} of $f$ and $f$ is said to be $K-$quasiconformal if $\mathcal{K}(f)\le K$.

The relation between quasiconformal maps and quasi-symmetric homeomorphism is classical, thanks to \cite[Theorem~1]{ahlbeu}, which states: 
\begin{thm}[Ahlfors-Beuring]
	Let $f\colon\hyp^2\to\hyp^2$ be a quasiconformal map, then $f$ extends to a unique quasi-symmetric homeomorphism of $\pd_\infty\hyp^2=\pd_\infty\hyp^2$.
	
	Conversely, any quasi-symmetric homeomorphism $\phi\colon\pd_\infty\hyp^2\to\pd_\infty\hyp^2$ admits an extension to the disc which is quasiconformal in the interior.
\end{thm}
The quasiconformal extension is far from being unique: it can be useful to choose a specific class of quasiconformal maps to build a bijective correspondence with the universal Teichm\"uller space. We focus on landslides, which have been introduced in \cite{theta} as smooth version of earthquakes.
\begin{de}\label{de:theta}
	Let $\theta\in(0,\pi)$. An orientation preserving diffeomorphism $f\colon\hyp^2\to\hyp^2$ is a \textit{$\theta-$landslide} if there exists a $(1,1)-$tensor $m\in\Gamma\left(\mathrm{End}(T\hyp^2)\right)$
	such that 
	\[f^*g_{\hyp^2}=g_{\hyp^2}\left(\cos(\theta)\Id+\sin(\theta)Jm,\cos(\theta)\Id+\sin(\theta)Jm\right),\]
	for $J$ the complex structure of $\hyp^2$, and
	\begin{enumerate}
		\item $d^\nabla m=0$;
		\item $\det m=1$;
		\item $m$ is positive and self-adjoint for $g_{\hyp^2}$.
	\end{enumerate}
\end{de}

\subsection{Gauss map}\label{sub:gauss} The bridge between diffeomorphisms of the hyperbolic space and spacelike surfaces in the $3-$dimensional Anti-de Sitter space is given by the \textit{Gauss map}, whose construction is briefly described in the following. A refined presentation of the Gauss map can be found in \cite{bonsep}. 

The space of timelike geodesic of $\hyp^{2,1}$ naturally identifies with $\hyp^2\times\hyp^2$ (see for example \cite[Lemma~3.5.1]{bonsep}). Hence, embedded spacelike surfaces in $\hyp^{2,1}$ induce immersed surfaces in $\hyp^2\times\hyp^2$: indeed, let $\Sigma$ be an embedded spacelike surface in the Anti-de Sitter space: any point $x\in\Sigma$ identifies one preferred timelike geodesic, namely $\exp_x(\R N_x\Sigma)$, which corresponds to a point $\mathcal{G}(x)\in\hyp^2\times\hyp^2$.

Let $p_l,p_r\colon\hyp^2\times\hyp^2\to\hyp^2$ the projection respectively on the first and second coordinate. Denote by $\mathcal{G}_i=\mathcal{G}\circ p_i$ the projection of the Gauss map $\mathcal{G}$ on each factor. 

\begin{lem}[Lemma~4.2 in \cite{areapres}]
	Let $\Sigma$ be convex spacelike surface, then $\mathcal{G}_l,\mathcal{G}_r$ are injective, hence they induce a homeomorphism between open subsets of $\hyp^2$. 
	\[\Phi_\Sigma:=\mathcal{G}_r\circ\mathcal{G}_l^{-1}\colon\mathcal{G}_l(\Sigma)\to\mathcal{G}_r(\Sigma)\]
\end{lem}

This fact has been exploited in order to study homeomorphisms of the hyperbolic space: the pioneer of this tradition is \cite{mess}, where convex pleated spacelike surfaces are proved to correspond to earthquakes. In \cite{univ}, Gauss map links maximal surfaces in $\hyp^{2,1}$ and minimal Lagrangian diffeomorphisms of $\hyp^2$. The latter results is particular case of a bigger picture, studied in \cite{areapres}: surfaces with constant sectional curvature $K$ (which are convex by Gauss equation) correspond to $\theta-$landslides, for $\theta=2\arccos(\sqrt{-1/K})$ (see \cite[Proposition~4.13]{areapres}). In particular, $(\pi/2)-$landslide are minimal Lagrangian diffeomorphisms and maximal surfaces are equidistant to surfaces with constant sectional curvature $K=-2$ (Lemma~\ref{lem:HKduality}), hence they induce the same Gauss map. 

In $\hyp^{2,1}$, CMC surfaces are equidistant to surfaces with constant sectional curvature, as proved in the following lemma. 
	\begin{lem}\label{lem:HKduality}
	Let $\Sigma$ be a $H-$surface, with $H\in\R$. Then, the equidistant surface $\Sigma_{d(H)_\pm}$ is a $K_\pm(H)-$surface, for
	\[d(H)_\pm=\arctan\left(\frac{H}{2}\pm\sqrt{1+\frac{H^2}{4}}\right),\qquad K_\pm(H)=-1-\frac{4}{(H\pm\sqrt{1+H^2})^2}.\]
	
	Conversely, let $S$ be a future-convex (resp. past-convex) $K-$surface, for $K\in(-\infty,-1)$. Then, the equidistance $S_{d(K)}$ (resp. $S_{-d(K)}$) is a $H(K)-$surface (resp. $(-H(K))-$surface), for
	\begin{equation}\label{eq:dKHK}
		d(K)=\arctan\left(\frac{1}{\sqrt{-1-K}}\right),\qquad H(K)=\frac{2-K}{\sqrt{-1-K}}.
	\end{equation}
\end{lem}
\begin{proof}
	The proof consists in an explicit (local) compution carried in light of Corollary~\ref{cor:principalcurvature}. 
\end{proof}
By definition, Gauss map is invariant along the normal flow: it follows by \cite[Proposition~4.13]{areapres} and Lemma~\ref{lem:HKduality} that $H-$surfaces induce $\theta-$landslide, as well.
\begin{cor}\label{cor:thetaH}
	Let $\Sigma$ be a properly embedded $H-$surface in $\hyp^{2,1}$, then the diffeomorphism $\Phi_\Sigma$ induced by the Gauss map is a $\theta-$landslide, for
	\[\theta=2\arccos\left(\frac{(H+\sqrt{4+H^2})}{\sqrt{{(H+\sqrt{4+H^2})^2}+4}}\right).\]
\end{cor}

\subsection{Quasiconformal dilatation} In the same spirit of \cite{andreamax}, we compare the quasiconformal dilatation of a $\theta-$landslide with the principal curvature of the corresponding $H-$surface. In fact, let $\Sigma$ be a properly embedded $H-$surface in $\hyp^{2,1}$, and let $a\in C^0(\Sigma)$ be the non-negative eigenvalue of the traceless shape operator of $\Sigma$, that is
\[B(x)=\begin{pmatrix}
	\frac{H}{2}+a(x) & \\
	& \frac{H}{2}-a(x)
\end{pmatrix}\] in a suitable orthonormal basis of $T_x\Sigma$. Let $x\in\Sigma$, we can apply \cite[Proposition 6.2]{tamb} to get
\begin{equation}\label{eq:dilatation}
	\mu_{\Phi_{\Sigma}}\left(\Gau_\Sigma(x)\right)=-a(x)\frac{(H/2)+i}{1+(H/2)^2},
\end{equation}
for $\mu_{\Phi_{\Sigma}}$ the complex dilatation of the $\theta-$landslide $\Phi_\Sigma$ associated to $\Sigma$.

In light of Theorem~\ref{thm:Schauder}, this allows to estimate the quasiconformal dilatation with the cross-ratio norm of its extension to the boundary.
\begin{repthmx}{thm:qconformal}
	For any $\alpha\in(0,\pi/2)$, there exist universal constants $Q_\alpha,\eta_\alpha>0$ such that 
	\[
	\ln\left(\mathcal{K}(\Phi_\theta)\right)\le Q_\alpha\|\phi\|_{cr},
	\]
	for $\Phi_\theta$ the only $\theta-$landslide extending $\phi$, with $\theta\in[\alpha,\pi-\alpha]$ and $\|\phi\|_{cr}\le\eta_\alpha$.
\end{repthmx}
\begin{proof}
	Let us fix $\theta\in[\alpha,\pi-\alpha]$, and a quasi-symmetric homeomorphism $\phi$. Let $\Phi_\theta\colon\hyp^2\to\hyp^2$ be the unique $\theta-$landslide extending $\phi$ and denote by $\Sigma$ the corresponding $H-$surface, \textit{i.e.} $\Sigma$ is the unique $H-$surface bounded by $\Lambda=\gr\phi$. By Equation~\eqref{eq:dilatation}, the maximal dilatation of $\Phi_\theta$ at $z=\Gau_\Sigma(x)$ is 
	\[\left|\mu_{\Phi_\theta}\left(\Gau_\Sigma(x)\right)\right|^2=\frac{a(x)^2}{1+(H/2)^2}.\]
	By substituting in Equation~\eqref{eq:maxdil}, the maximal dilatation of $\Phi_\theta$ is
	\begin{equation}\label{eq:dil1}
		\mathcal{K}\left(\Phi_\theta\right)=\frac{1+|\mu_{\Phi_\theta}|^2_{L^{\infty}(\hyp^2)}}{1-|\mu_{\Phi_\theta}|^2_{L^{\infty}(\hyp^2)}}=\frac{1+(H/2)^2+\|B_0\|^2_{C^0(\Sigma)}}{1+(H/2)^2-\|B_0\|^2_{C^0(\Sigma)}}.
	\end{equation}
	The function $\theta=\theta(H)$ introduced in Corollary~\eqref{cor:thetaH} is a proper diffeomorphism, hence there exists $L\ge0$ such that 
	\[\theta^{-1}\left([\alpha,\pi-\alpha]\right)\sq[-L,L].\]
	In particular, combining Theorem~\ref{thm:Schauder}, Lemma~\ref{lem:width} and Lemma~\ref{lem:cross} we obtain	
	\begin{align*}
		\|B_0\|^2_{C^0(\Sigma)}&\le C_L\sin\left(\omega_H(\pd_\infty\Sigma)\right)\le C_L\sin\left(\omega_0(\pd_\infty\Sigma)\right)\\
		&\le C_L\tan\left(\omega_0(\pd_\infty\Sigma)\right)\le C_L\sinh\left(\frac{\|\phi\|_{cr}}{2}\right).
	\end{align*}
	The last formula in Equation~\eqref{eq:dil1} is increasing in $\|B_0\|^2_{C^0(\Sigma)}$, hence
	\begin{align*}\mathcal{K}(\Phi_\theta)&\le\frac{1+(H/2)^2+C_L\sinh\left(\frac{\|\phi\|_{cr}}{2}\right)}{1+(H/2)^2-C_L\sinh\left(\frac{\|\phi\|_{cr}}{2}\right)}\\
		&\le\frac{1+C_L\sinh\left(\frac{\|\phi\|_{cr}}{2}\right)}{1-C_L\sinh\left(\frac{\|\phi\|_{cr}}{2}\right)},\end{align*}
	since the function is decreasing in $H^2$. The inequality is valid by setting $\eta_\alpha<2\arcsinh(C_L^{-1})$, which ultimately depends on $\alpha$ since $L$ depends on it.
	
	Finally, we remark that
	\[\frac{d}{dx}\ln\left(\frac{1+C_L\sinh(x)}{1-C_L\sinh(x)}\right)=\frac{2C_L\cosh(x)}{1-C_L^2\sinh^2(x)},\]
	hence there is a constant $Q_\alpha>0$ depending on $C_L,\eta_\alpha$, hence ultimately only on $\alpha$, such that
	\[\ln\left(\frac{1+C_L\sinh\left(\frac{\|\phi\|_{cr}}{2}\right)}{1-C_L\sinh\left(\frac{\|\phi\|_{cr}}{2}\right)}\right)\le Q_\alpha\|\phi\|_{cr},\]
	which concludes the proof.
\end{proof}

\printbibliography

\end{document}